\documentclass[12pt]{amsart}
\usepackage{amsmath}
\usepackage{amsfonts}
\usepackage{amsthm}
\usepackage{amssymb}
\usepackage{amscd}
\usepackage[all]{xy}
\usepackage{enumerate}


\theoremstyle{plain}
\newtheorem{thm}{Theorem}[section]

\newtheorem{prop}[thm]{Proposition}

\newtheorem{cor}[thm]{Corollary}

\newtheorem{lem}[thm]{Lemma}

\theoremstyle{definition}
\newtheorem{defn}[thm]{Definition}

\newtheorem*{ackn}{Acknowledgment}

\newtheorem{rmk}[thm]{Remark}
\newcommand{\sA}{\mathcal{A}}

\newcommand{\sD}{\mathcal{D}}
\newcommand{\sE}{\mathcal{E}}
\newcommand{\sF}{\mathcal{F}}

\newcommand{\sH}{\mathcal{H}}
\newcommand{\sI}{\mathcal{I}}

\newcommand{\sK}{\mathcal{K}}
\newcommand{\sL}{\mathcal{L}}

\newcommand{\sO}{\mathcal{O}}

\newcommand{\sU}{\mathcal{U}}

\newcommand{\sW}{\mathcal{W}}

\newcommand{\mC}{\mathbb{C}}
\newcommand{\mD}{\mathbb{D}}

\newcommand{\mW}{\mathbb{W}}

\newcommand{\Ker}{\mathrm{Ker}\,}

\usepackage{color}
\numberwithin{equation}{section}

\newcommand{\beba}  {\begin{equation}\begin{array}{rcl}}

\newcommand{\eaee}  {\end{array}\end{equation}}

\makeatletter
\def\l@section{\@tocline{1}{0pt}{1pc}{}{}}
\def\l@subsection{\@tocline{2}{0pt}{1pc}{4.6em}{}}
\def\l@subsubsection{\@tocline{3}{0pt}{1pc}{7.6em}{}}
\renewcommand{\tocsection}[3]{%
  \indentlabel{\@ifnotempty{#2}{\makebox[2.3em][l]{%
    \ignorespaces#1 #2.\hfill}}}#3}
\renewcommand{\tocsubsection}[3]{%
  \indentlabel{\@ifnotempty{#2}{\hspace*{2.3em}\makebox[2.3em][l]{%
    \ignorespaces#1 #2.\hfill}}}#3}
\renewcommand{\tocsubsubsection}[3]{%
  \indentlabel{\@ifnotempty{#2}{\hspace*{4.6em}\makebox[3em][l]{%
    \ignorespaces#1 #2.\hfill}}}#3}
\makeatother

\title[Massey products and Fujita decomposition over higher dimensional base]{Massey products and Fujita decomposition over higher dimensional base}

\author{Luca Rizzi}
\address{Luca Rizzi\\Department of Mathematics, Computer Science and Physics \\
	Universit\`a di Udine\\
	Udine, 33100\\ Italia
	\texttt{luca.rizzi@uniud.it}}

\begin{document}

\begin{abstract} Let  $f\colon X\to Y$ be a semistable fibration between smooth complex  varieties of dimension $n$ and $m$. This paper contains an analysis of the local systems of de Rham closed relative one forms and top forms on the fibers. 
In particular the latter recovers the local system of the second Fujita decomposition of $f_*\omega_{X/Y}$ over higher dimensional base.
The so called theory of Massey products allows, under natural Castelnuovo-type hypothesis, to study the finiteness of the associated monodromy representations. Motivated by this result, we also make precise the close relation between Massey products and Castelnuovo-de Franchis type theorems. 
\end{abstract}

\maketitle
%\tableofcontents
%%%%%%%%%%%%%%%%%%%%%%%%%%%%%%%%%%%%%%%%%%%%%%%%%%%%%%%
%%%%%%%%%%%%%%%%%%%%%%%%%%%%%%%%%%%%%%%%%%%%%%%%%%%%%%%
%%%%%%%%%%%%%%%%%%%%%%%%%%%%%%%%%%%%%%%%%%%%%%%%%%%%%%%
\section{Introduction}

In this paper $f\colon X\to Y$ will denote a proper surjective morphism between complex smooth varieties of dimension $n$ and $m$ respectively and with general smooth fiber denoted by $F$ or, if we want to make the base point $y\in Y$ explicit, $F_y$. We mainly focus on \emph{semistable fibrations} according to the definition given in \cite{Il}, \cite{MR}, see Section \ref{sez1} for the details.

If $Y=B$ is a smooth curve, a famous result of Fujita, see \cite{Fu} and \cite{Fu2}, states that the direct image of the relative dualizing sheaf $f_*\omega_{X/B}$ has two splittings, the \emph{first and second Fujita decomposition}. These are respectively   
\begin{equation}
f_*\omega_{X/B}\cong \sO_B^h\oplus \sE
\end{equation} where $\sE$ is a locally free nef sheaf on $B$ with $h^1(B,\sE\otimes \omega_B)=0$
and
\begin{equation}
f_*\omega_{X/B}\cong \sU\oplus \sA
\end{equation} where $\sU$ is a unitary flat vector bundle and $\sA$ ample. 

This result has been later generalized in different ways, and in this paper we are mostly concerned with the case where the dimension of $Y$ is greater than 1. This is studied in \cite{CK} where the authors prove that the second Fujita decomposition over higher dimensional base $Y$ is as follows:
\begin{equation}
\label{fujitaCK}
f_*\omega_{X/Y}\cong \sU\oplus \sW
\end{equation} where $\sU$ is locally free and unitary flat with respect to the natural hermitian connection and $\sW$ is generically ample, i.e., $\sW_{|C}$ is an ample vector bundle for any general curve section $C$ of $Y$.

Now recall that there is a one to one correspondence modulo isomorphism between flat vector bundles on $Y$, local systems of $\mC$ vector spaces on $Y$ and representations of the fundamental group $\pi_1(Y,y)$.
Hence, naturally associated to $\sU$ there are also a local system and a representation.
In the case of $Y=B$ a curve, an analysis  of $\sU$ and its associated monodromy is in \cite{PT} and \cite{RZ4}. In this paper, motivated by the above mentioned \cite{CK}, we deal with the general case where $\dim Y>1$.

We study $\sU$ and the associated monodromy thanks to the theory of \emph{Massey products}. Massey products have been introduced in \cite{CP} and \cite{PZ} and then applied in \cite{Ra}, \cite{PR}, \cite{CNP}, \cite{victor}, \cite{BGN}, 
\cite{RZ2}, \cite{RZ3}, \cite{RZ1}, \cite{CRZ} and \cite{R}, hence we refer to this sources for a complete discussion and here we give only the general idea of the construction.
In the case of one dimensional base $Y=B$, consider the $(n-1)$-dimensional fiber $F$ and take 1-forms,  $\eta_1,\dots,\eta_n$, in the kernel of the cup product $\cup \xi\colon H^0(\Omega^1_{F})\to H^1(\sO_F)$ where $\xi$ is the associated infinitesimal deformation, $\xi\in H^1(T_F)$.
By the exact sequence 
\begin{equation}
	\label{solitain}
	0\to \sO_{F}\to \Omega^1_{X|F}\to \Omega^1_{F}\to 0
\end{equation} these sections can be lifted to $H^0(\Omega^1_{X|F})$.
 Choosing $s_1,\dots,s_n\in H^0(\Omega^1_{X|F})$ liftings of the $\eta_i$ we have a top form $\Omega\in H^{0}(\omega_{X|F})$ from the element $s_{1}\wedge\ldots\wedge s_{n}$. Since $\omega_{X|F}\cong \omega_{F}$ we obtain from $\Omega$ a top form of the canonical sheaf $\omega_{F}$. Such a form is the classical \emph{adjoint form} or \emph{Massey product} of $\eta_1,\dots,\eta_n$. We usually denote it by $m_{\xi}(\eta_1,\ldots,\eta_n)$. 

We are actually interested in the condition called Massey triviality: we say that the sections $\eta_1,\dots,\eta_n$ are Massey trivial if their Massey product $m_{\xi}(\eta_1,\ldots,\eta_n)$ is a linear combination of the top forms $\eta_1\wedge\dots\wedge\widehat{\eta_i}\wedge\dots\wedge\eta_n$, $i=1\dots,n$. The Massey product $m_{\xi}(\eta_1,\ldots,\eta_n)$ depends on the choice of the liftings $s_i$, but it turns out that the condition of Massey triviality does not. If furthermore the forms $\eta_1\wedge\dots\wedge\widehat{\eta_i}\wedge\dots\wedge\eta_n$ are linearly independent in $H^0(\omega_F)$ we say that the $\eta_i$ form a \emph{strict} subspace of $H^0(\Omega^1_F)$. 

In this paper the above construction is not directly helpful for two reasons. First we have to deal with the case where the base is of dimension strictly greater than $1$. Furthermore, since we want to study a family $X\to Y$ and not just the neighborhood around a fiber $F$, it is not enough to construct Massey products on single fibers independently as described above, but we have to work in families, using the ideas of \cite{PT}, \cite{RZ4}.

The solution to both these problems is studied in Section \ref{sezioneaggiunta}.
The idea is to take the pushforward via $f$ of the exact sequence 
\begin{equation}
\label{diffrelint}
0\to f^*\Omega^1_Y\to \Omega^1_X\to \Omega^1_{X/Y}\to 0
\end{equation}
where $\Omega^1_{X/Y}$ is the sheaf of relative differentials. Now the connecting morphism 
$$
\partial \colon f_*\Omega^1_{X/Y}\to R^1f_*\sO_X\otimes\Omega^1_Y
$$ restricted on the general fiber is exactly the cup product $\cup\xi$, and we will denote its kernel by $K_\partial:=\ker\partial$. This means that a local section of $K_\partial$ on a suitable open subset $A$ of $Y$ can be regarded as a collection of liftable 1-forms $\{\eta_{y}\}_{y\in A}$ on each fiber $F_y$. One of the central points of this construction is that such sections of $K_\partial$ can always be lifted to $f_*\Omega^1_X$; this is proved in Lemma \ref{split}. These liftings allow the definition of a new notion of Massey product which works both over higher dimensional base and in families. Of course the notion of Massey triviality and strictness also extended nicely, see Section \ref{sezioneaggiunta} for details. 
%In Section \ref{sezioneaggiunta} we also show how to work with a higher dimensional base and construct, starting from sections of $K_\partial$, a Massey product.  $s_i$ of the sections $\eta_i$, now  
%and by choosing $n$ local sections of $K_\partial$ we can construct a family of Massey products which roughly corresponds to gluing together the $m_{\xi}(\eta_1,\ldots,\eta_n)$ for all the fibers over the considered open subset. 

Actually it turns out that $K_\partial$ contains a local system $\mD$ and the study of Massey products from sections of $\mD$ is the right approach to monodromy problems.

In the case $d:=n-m=1$, that is when $F$ is a curve, the local system $\mD$ is actually the local system univocally associated to the flat vector bundle $\sU$ of the second Fujita decomposition (\ref{fujitaCK}). In the general case there is still a close relation between $\mD$ and the local system associated to $\sU$, denoted by $\mD^{d}$. In fact by the analysis contained in Section \ref{sez1}, see also \cite{RZ4}, both are local systems related to closed holomorphic differential forms on $X$ as follows.  The local systems $\mD$ and $\mD^{d}$ can be interpreted respectively as the sheaves of 1-forms and top forms  on the fibers which can be lifted to \textit{closed} holomorphic forms on  $X$. More precisely $\mD$ is the image of $f_*\Omega^1_{X,d}$ in the sheaf of closed relative differential 1-forms $f_*\Omega^1_{X/Y,d_{X/Y}}$ and similarly $\mD^d$ is the image of $f_*\Omega^d_{X,d}$ in the sheaf of closed relative differential $d$-forms $f_*\Omega^d_{X/Y,d_{X/Y}}$.

Now let $A\subset B$ be a contractible open subset and $W\subset \Gamma(A, \mD)$ a vector subspace of dimension at least $d+1$. We say that $W$ is \emph{Massey trivial} if any $d+1$-uple of sections in $W$ is Massey trivial, see Definition \ref{mastriv}.
Furthermore we say that a local subsystem $\mW\leq\mD$ is \emph{Massey trivial generated} if its general fiber is generated under the monodromy action by a Massey trivial vector space $W$, see Definition \ref{mastrivgen}.
In Section \ref{sez4} we show that these conditions of Massey triviality are strictly related to a theorem by Castelnuovo and de Franchis, recalled in Theorem \ref{cas2}.

 Section \ref{sez5} contains the study of the monodromy of a Massey trivial generated local system $\mW$. Call $\rho_\mW $ the action of the fundamental group $\pi_1(Y, y)$ on the stalk of $\mW$ and call $G_\mW=\pi_1(Y, y)/\ker \rho_\mW$ the monodromy group. 
We construct a faithful action of this group on a suitable set $\sK$ of surjective morphisms from the general fiber $F$ to a normal variety $Z$ of general type and since it is well known that the number of such maps is finite, $\sK$ is finite and we can prove the finiteness of the monodromy.
%\begin{thml}
%\label{B}
%Let $f \colon X \to B$ be a semistable fibration on a smooth projective curve B and let $\mW<\mD$ be a strict Massey trivial generated local system.
%Then the associated monodromy group $G_\mW$ is finite and the fiber of $\mW$ is isomorphic to 
%$$
%\sum_{k\in \sK} k^*H^0(Y,\Omega^1_Y).
%$$
%%Where $h\colon X_{\mW}\to Y$ is the higher irrational pencil associated to ${\mW}$.
%\end{thml}
\begin{thm}
	\label{A}
Let $f \colon X \to Y$ be a semistable fibration and let $\mW\leq\mD$ be a  local system generated by a maximal strict Massey trivial subspace.
Then the associated monodromy group $G_\mW$ is finite and the fiber of $\mW$ is isomorphic to 
$$
\sum_{k\in \sK} k^*H^0(Z,\Omega^1_Z).
$$
%Where $h\colon X_{\mW}\to Y$ is the higher irrational pencil  associated to ${\mW}$.
\end{thm}
See Theorem \ref{monfin}.

As a corollary, see Corollary \ref{semi}, we obtain the following result on the monodromy of $\mD$ and $\mD^{d}$
\begin{cor}
If $\mD$ is Massey trivial generated by a strict subspace, then its monodromy group is finite.
If furthermore the map $\bigwedge^{d}\mD\to \mD^{d}$ is surjective, the local system $\mD^{d}$ also has finite monodromy. 
\end{cor}
See Corollary \ref{hyper} for an example where this is applied.

 Another result on the monodromy of $\mD^d$ and its local subsystems comes from the notion of $p$-strictness.
Let $X$ be a smooth variety, $w_1,\dots, w_l \in H^0 (X,\Omega^p_X)$, $l\geq p+1$, be linearly independent $p$-forms such that $w_i\wedge w_j=0$ (as an element of $\bigwedge^2\Omega^p_X$ and not of $\Omega_X^{2p}$) for any choice of $i,j=1,\dots, l$. These forms generate a subsheaf of $\Omega^p_X$ generically of rank $1$. Note that the quotients $w_i/w_j$ define a non-trivial global meromorphic function on $X$ for every $i\neq j$, $i,j=1,\dots, l$. By taking the differential $d (w_i/w_j)$ we then get global meromorphic $1$-forms on $X$. We assume that there exist $p$ of these meromorphic differential forms $d (w_i/w_j)$ that do not wedge to zero; if this is the case we call  the subset $\{w_1,\dots, w_l \}\subset H^0 (X,\Omega^p_X)$ $p$-\textit{strict}.  This notion is of course in some sense the generalization for $p$-forms of the notion of strictness seen above. It actually allows to prove a version of the Castelnuovo-de Franchis theorem for $p$-forms, with $p>1$, see \cite[Theorem 7.2]{RZ5}. This in turn gives the following result.

\begin{thm}
	Let $W=\langle\eta_1,\dots,\eta_l\rangle<\Gamma(A,\mD^d)$ and assume that the sections $\eta_{i}$ are $d$-strict and admit liftings $s_i\in \Gamma(A,f_*\Omega_{X,d}^d)$ with $s_i\wedge s_j=0$ for every choice of $i,j$. Then $W$ generates a local system $\mW\leq\mD^d$ with finite monodromy group $G_\mW$.
\end{thm}

Recall that the finiteness of the monodromy group of a local system is equivalent to the semi-ampleness of the unitary flat vector bundle, see for example \cite[Theorem 2.5]{CD1}, hence since $\mD^{d}$ is the local system associated to $\sU$, these results are indeed a tool to study the semi-ampleness of $\sU$.

When working with Massey products, both in the above mentioned literature and here in this paper up until this point, a crucial step consists in taking a wedge of $d+1$ differential forms on $X$, that is exactly one above the dimension of the fiber $F$. In Section \ref{masseybig} we consider a wedge product of $k>d+1$ sections instead. We point out that in the usual context of Massey products of \cite{RZ1}, \cite{PT}, \cite{RZ4}, that is when the base $Y$ is a curve, we have that $\dim X=d+1=n$. Hence this analysis is only meaningful when $Y$ is not a curve, and it nicely fits into the scope of this paper. 

We find the natural definition of $k$-Massey products and $k$-Massey triviality, see Definitions \ref{kmas} and \ref{kmt}, and we show that this theory is once again very closely related to the theory of the Castelnuovo-de Franchis Theorem, see Theorem \ref{castmassey2}.

We prove that $k$-Massey triviality gives a result similar to Theorem \ref{A} if we further assume that the covering $Y_\mW\to Y$ associated to the monodromy group $G_\mW$ is such that $Y_\mW\subset \overline{Y_\mW}$ is a Zarisky open subset of a compact variety.
\begin{thm}
	Let $f \colon X \to Y$ be a semistable fibration and let $\mW\leq\mD$ be a local system generated by a strict $k$-Massey trivial subspace. Assume also that $Y_\mW\subset \overline{Y_\mW}$ is a Zarisky open subset of a compact variety.
	Then the associated monodromy group $G_\mW$ is finite.
\end{thm} 
See Theorem \ref{chiusozar}.

Thanks to the relation between Massey triviality and Castelnuovo-de Franchis theorems, we prove a natural generalization of the Adjoint theorem of \cite{PZ}, \cite{RZ1} which gives us information on some extension groups on the fibers. See Theorem \ref{aggiuntanuovo} and Corollary \ref{cork}.

%Take a general point $y\in Y$, we  restrict Sequence \ref{diffrelint} to the fiber $F=F_y$ over $y$ and obtain the exact sequence
%\begin{equation}
%0\to \sO_F^m\to \Omega^1_{X|F}\to \Omega^1_F\to 0.
%\end{equation}  This sequence is extensively studied in \cite{RZ1} with $m=1$. We have that 
%\begin{thm}
%Consider the sequence above and $\eta_1,\dots,\eta_k$ global sections of $H^0(\Omega_F^1)$ which are also in $\mD$. If the $\eta_i$ are $k$-Massey trivial, then there exists a subsheaf $\sF$ strictly contained in $\Omega^1_{X|F}$ which maps surjectively onto $ \Omega^1_{F}(-D)$, where $D$ is an effective divisor defined as the the fixed part of the sections $\eta_{j_1}\wedge\dots\wedge\widehat{\eta_{j_i}}\wedge\dots\wedge\eta_{j_n}$ with $\{j_1,\dots, j_n\}\subset\{1,\dots,k\}$.
%If we take $k=d+1$, $\sF$ is isomorphic to $ \Omega^1_{F}(-D)$ and we have that $\xi\in {\rm{Ext}}^1(\Omega^1_F,\sO_F^m)$ corresponding to Sequence \ref{seqm} is in the kernel of 
%\begin{equation}
%{\rm{Ext}}(\Omega^1_F,\sO_F^m)\to {\rm{Ext}}(\Omega^1_F(-D),\sO_F^m).
%\end{equation}
%\end{thm}
%See Theorem \ref{aggiuntanuovo} and Corollary \ref{cork}. This intuitively says that, up to the divisor $D$, the deformation data given by 
%$$
%0\to \sO_F^m\to \Omega^1_{X|F}\to \Omega^1_F\to 0.
%$$ is enclosed in an exact sequence of sheaves of smaller rank. Note also that the result is particularly significative when $D$ is zero, that is the sections $\eta_{j_1}\wedge\dots\wedge\widehat{\eta_{j_i}}\wedge\dots\wedge\eta_{j_n}$ do not vanish on a common divisor.
\begin{ackn}
	This work was supported by JSPS-Japan Society for the Promotion of Science (Postdoctoral Research Fellowship, The University of Tokyo), by the IBS Center for Complex Geometry, Daejeon, South Korea and by European Union funds, NextGenerationEU.
\end{ackn}
\section{Semistable Fibrations over higher dimensional base and differential forms}
\label{sez1}

In this paper we consider \emph{semistable fibrations} according to the definition given by Illusie in \cite{Il}, see also \cite{MR}.

Let $X$ and $Y$ be smooth complex compact varieties of dimension $n$ and $m$ respectively and $f\colon X\to Y$ a surjective morphism. Let $D$ be a normal crossing divisor on $X$ and $E$ a normal crossing divisor on $Y$ such that $D=f^{-1}(E)$. Take local coordinates $x_1,\dots,x_n$ on $X$ such that $D$ is locally given by $x_1\ldots x_k=0$ and local coordinates $y_1,\dots,y_m$ on $Y$ such that $E$ is $y_1\ldots y_p=0$.
Under these assumptions, the semistable morphism $f$ is given locally by 
\begin{equation}
f(x_1,\dots,x_n)=(x_1\cdots x_{s_1},x_{s_1+1}\cdots x_{s_2},\dots,x_{s_{p-1}+1}\cdots x_{s_p},x_{s_p+t_1},x_{s_p+t_2},\dots,x_{s_p+t_{m-p}})
\label{semistable}
\end{equation} with $s_p=k$ and $1\leq t_1\leq t_2\leq\dots\leq t_{m-p}\leq n-k$.
In particular the fibers of $f$ are reduced.

As in the case of fibrations over a curve considered in \cite{PT} and \cite{RZ4}, the \emph{relative dualizing sheaf} $\omega_{X/Y}$  and the \emph{relative differentials} $\Omega^p_{X/Y}$ will play an important role in the following. 

Recall that the relative dualizing sheaf $\omega_{X/Y}$ is 
\begin{equation}
\omega_{X/Y}:=\omega_X\otimes f^*\omega_Y^\vee
\label{reldual}
\end{equation} and it is locally free since both $X$ and $Y$ are smooth. Furthermore it is well known that in the case of a semistable fibration its higher direct images $R^pf_*\omega_{X/Y}$ are also locally free on $Y$, see for example \cite{Il}; we will be interested mainly in $f_*\omega_{X/Y}=f_*\omega_X\otimes \omega_Y^\vee$.

The exact sequence 
\begin{equation}
\label{diffrel}
0\to f^*\Omega^1_Y\to \Omega^1_X\to \Omega^1_{X/Y}\to 0 
\end{equation} defines the sheaf of relative differentials $\Omega^1_{X/Y}$. This sheaf is not locally free for a general fibration $f$, but it turns out to be at least torsion free in our setting. This can be easily seen as follows.

Recall that since $X$ is a smooth variety and $D$ is a normal crossing divisor on $X$, given locally by $x_1x_2\cdots x_k=0$, we can define the sheaf $\Omega^1_X(\text{log }D)$ of \emph{logarithmic differentials} as the locally free $\sO_X$-module generated by $dx_1/x_1,\ldots,dx_k/x_k,dx_{k+1},\ldots,dx_{n}$. In the same way we define the locally free $\sO_Y$-module $\Omega^1_Y(\text{log }E)$.
%This sheaf fits into the exact sequence
%\begin{equation}
%0\to \Omega^1_X\to \Omega^1_X(\text{log }D) \stackrel{\text{Res}}{\rightarrow} \bigoplus_i \sO_{D_i}\to 0
%\end{equation} where $D_i$ are the irreducible components of $D$ and $\text{Res}$ is the residue map.

We have an injection 
$
f^*\Omega^1_Y(\text{log }E)\to \Omega^1_X(\text{log }D)
$ with locally free cokernel which is denoted by $\Omega^1_{X/Y}(\text{log})$:
\begin{equation}
\label{logdiffrel}
0\to f^*\Omega^1_Y(\text{log }E)\to \Omega^1_X(\text{log }D)\to \Omega^1_{X/Y}(\text{log})\to 0.
\end{equation}
Now call $D_i$ the irreducible components of $D$ and $E_u$ the irreducible components of $E$. Call $\sO_{D_0}=\oplus_i \sO_{D_i}$ and $\sO_{E_0}=\oplus_u\sO_{E_u}$, then we have the commutative diagram
\begin{equation}
\xymatrix{
&0\ar[d]&0\ar[d]&0\ar[d]&\\
0\ar[r]& f^*\Omega^1_Y\ar[d]\ar[r]&\Omega^1_X\ar[d]\ar[r]&\Omega^1_{X/Y}\ar[r]\ar[d]&0\\
0\ar[r]& f^*\Omega^1_Y(\text{log }E)\ar[d]\ar[r]&\Omega^1_X(\text{log }D)\ar[d]\ar[r]&\Omega^1_{X/Y}(\text{log})\ar[r]\ar[d]&0\\
0\ar[r]& f^*\sO_{E_0}\ar[r]\ar[d] &\sO_{D_0}\ar[r]\ar[d]&\sO_{D_0}/f^*\sO_{E_0}\ar[r]\ar[d]&0\\
&0&0&0&
}
\label{diagrammalog}
\end{equation} see \cite{MR}.
%About the arrows in the diagram, note that 
%\begin{enumerate}
%	\item This injectivity comes from the flatness of $f$.
%	\item This arrow is the obvious one coming from the maps $\sO_{f^*E_u}\to \sO_{D_i}$ where $D_i\subset f^{-1}E_u$. It is injective because the fibers are reduced.
%	\item This arrow, which exists by commutativity, is injective because (2) is. 
%	\item This arrow also follows by the commutativity of the diagram.
%\end{enumerate}

In particular $\Omega^1_{X/Y}$ is torsion free because it is a subsheaf of a locally free sheaf.
The exterior powers $\Omega^p_{X/Y}=\bigwedge^p\Omega^1_{X/Y}$ are also torsion free, together with their direct images $f_*\Omega^p_{X/Y}$.

Finally call $d:=n-m=\dim X-\dim Y$ the relative dimension, we note that $\Omega^{d}_{X/Y}(\textnormal{log})=\omega_X\otimes f^*\omega_Y^\vee=\omega_{X/Y}$ by Sequence (\ref{logdiffrel}).
From this remark and the above Diagram (\ref{diagrammalog}), it is easily seen that we have an injection 
\begin{equation}
\label{injection}
f_*\Omega^{d}_{X/Y}\hookrightarrow f_*\omega_{X/Y}
\end{equation}
which is an isomorphism when restricted to $Y^0=Y\setminus E$. 

\subsection{Two subsheaves of $f_*\Omega^{1}_{X/Y}$}
It turns out that $f_*\Omega^{1}_{X/Y}$ contains two interesting subsheaves. They are studied in details in \cite{PT} and \cite{RZ4} when $Y$ is a curve, here we briefly recall their construction and highlight the main differences with the one-dimensional case.

%Over each regular value $b\in B^0$, Sequence (\ref{diffrel}) restricted to the fiber $F_b$ is the exact sequence
%\begin{equation}
%\label{solita}
%0\to T_{B,b}^\vee\otimes\sO_{F_b}\to \Omega^1_{X|F_b}\to \Omega^1_{F_b}\to 0 
%\end{equation} which gives an infinitesimal deformation of the fiber $F_b$, $\xi_b\in \text{Ext}^1(\Omega^1_{F_b},\sO_{F_b})\cong H^1(F_b, T_{F_b})$.
% We denote by $ \delta_{\xi_b}\colon H^0(X_b,\Omega^1_{X_b})\to H^1(X_b,  T_{B,b}^\vee\otimes\sO_{X_b}\simeq 
%H^1(X_b, \otimes\sO_{X_b}$ the associated connecting homomorphism.

Take the pushforward of Sequence (\ref{diffrel})
%\begin{equation}
%0\to f_*f^*\omega_B\to f_*\Omega^1_X\to f_*\Omega^1_{X/B}\to R^1f_*f^*\omega_B\to \dots
%\end{equation} which by projection formula is 
\begin{equation}
\label{seq1}
0\to \Omega^1_Y\to f_*\Omega^1_X\to f_*\Omega^1_{X/Y}\to R^1f_*\sO_X \otimes\Omega^1_Y\to \dots
\end{equation}
We call $F_y=f^{-1}(y)$ the fiber over a point $y\in Y$. 
Over each regular value $y\in Y^0$, Sequence (\ref{seq1}) gives
\begin{equation}
0\to T_{Y,y}^\vee\otimes H^0(\sO_{F_y})\to H^0(\Omega^1_{X|F_y})\to H^0(\Omega^1_{F_y})\stackrel{\delta_{\xi_y}}{\rightarrow} T_{Y,y}^\vee\otimes H^1(\sO_{F_y})\to \dots
\end{equation}
which is the cohomology long exact sequence of 
\begin{equation}
\label{solita}
0\to T_{Y,y}^\vee\otimes\sO_{F_y}\to \Omega^1_{X|F_y}\to \Omega^1_{F_y}\to 0.
\end{equation}
The first sheaf we introduce is denoted by $K_\partial$ and defined as follows:
\begin{defn}
The sheaf $K_\partial$ is the kernel of the map $\partial\colon f_*\Omega^1_{X/Y}\to R^1f_*\sO_X \otimes\Omega^1_Y$ of Sequence (\ref{seq1}).
\end{defn}
We note that, over the general $y\in Y^0$, 
\begin{equation}
\label{kerxi}
K_\partial\otimes \mC(y)= \ker \delta_{\xi_y}
\end{equation} that is, locally, we can think of $K_\partial$ as the sheaf of holomorphic one forms on the fibers of $f$ which are liftable to the variety $X$.
The key property of $K_\partial$ is that this liftability is not only local, as we see in the following lemma, 

\begin{lem}
\label{split}
Consider the exact sequence
\begin{equation}
0\to \Omega^1_Y\to f_*\Omega^1_X\to K_\partial\to 0.
\label{seqK}
\end{equation} 
If $A\subseteq Y$ is an open subset, all the sections of $\Gamma(A,K_\partial)$ can be lifted to $\Gamma(A,f_*\Omega^1_X)$.
\end{lem}
\begin{proof}
	See \cite[Lemma 3.5]{PT} and \cite[Lemma 2.2]{RZ4} for the case $\dim Y=1$; for $\dim Y>1$ the proof is a little more delicate, and we report it here.
	
Restrict Sequence (\ref{seqK}) on the open subset $A$ and note that it is enough to show that 
$$
H^1(A,\Omega^1_Y)\to H^1(A,f_*\Omega^1_X)
$$ is injective.
This comes from the fact that if we compose this map with the  map given by the Leray spectral sequence we obtain
$$
H^1(A,\Omega^1_Y)\to H^1(A,f_*\Omega^1_X)\to H^1(f^{-1}(A),\Omega^1_X)
$$ and it is enough to show that this composition is injective.

When $A=Y$ the maps $H^k(Y,\Omega^1_Y)\to H^k(X,\Omega^1_X)$ are injective for all $k$ by \cite[Lemma 7.28]{Vo1}.
In the case of $A\subset Y$ this remains true because the key point of this Lemma is that the general fiber $f^{-1}(y)$ is a non-zero cycle in $f^{-1}(A)$, hence the argument is the same.
\end{proof}

%\begin{rmk}
%\label{remark}
%Call $f^0\colon X^0\to Y^0$ the restriction of $f$ to the locus of regular values. Denote as usual by $\sH^{p,q}$ the Hodge bundles on $Y^0$. The map $\partial\colon f_*\Omega^1_{X/Y}\to R^1f_*\sO_{X} \otimes\Omega^1_{Y}$ on $Y^0$ is exactly the variation of Hodge structure
%\begin{equation}
%\overline{\nabla}^{1,0}\colon \sH^{1,0}\to \sH^{0,1}\otimes \Omega^1_Y
%\label{variation}
%\end{equation}See \cite[Section 10.2]{Vo1}.
%\end{rmk}

The second important subsheaf of $f_*\Omega^1_{X/Y}$ is a local system and it is defined as follows.
Consider the holomorphic de Rham complexes on $X$ and on $Y$ respectively:
\begin{equation}
0\to \mC_X\to \sO_X\to \Omega^1_{X}\to \Omega^2_{X}\to \dots\to \Omega^n_{X}\to 0
\label{DRX}
\end{equation} and 
\begin{equation}
0\to \mC_Y\to \sO_Y\to \Omega^1_{Y}\to \Omega^2_{Y}\to \dots\to \Omega^m_{Y}\to 0
\label{DRY}
\end{equation}
Now we compare the direct image of the short exact sequence on $X$
\begin{equation}
0\to \mC_X\to \sO_X\to \Omega^1_{X,d}\to 0
\label{1chiuse}
\end{equation} with the corresponding sequence on $Y$:
\begin{equation}
\xymatrix{
0\ar[r]& f_*\mC_X\ar @{=}[d]\ar[r]&f_*\sO_X\ar @{=}[d]\ar[r]&f_*\Omega^1_{X,d}\ar[r]&R^1f_*\mC_X\ar[r]&R^1f_*\sO_X\ar[r]&\dots\\
0\ar[r]&\mC_Y\ar[r]&\sO_Y\ar[r]&\Omega_{Y,d}^1\ar @{^{(}->}[u]\ar[r]&0
}
\label{dia1}
\end{equation}

\begin{defn}
The sheaf $\sD$ is defined as the cokernel of the vertical map $\Omega_{Y,d}^1\to f_*\Omega^1_{X,d}$. Alternatively by Diagram (\ref{dia1}) it is the image of $f_*\Omega^1_{X,d}\to R^1f_*\mC_X$ or the kernel of the map $R^1f_*\mC_X\to R^1f_*\sO_X$.
\end{defn}
 We have an inclusion of sheaves $\sD\hookrightarrow K_{\partial}$ and we can therefore interpret $\sD$ as the sheaf of  holomorphic one-forms on the fibers of $f$ which are liftable to \textit{closed} holomorphic forms of the variety $X$. 
 
%\begin{lem}
%\label{lemmainclusione}
%
%\end{lem}
%\begin{proof}
%Immediate by the fact that the exact sequence defining $\sD$
%\begin{equation}
%\label{seqdef}
%0\to \Omega_{Y,d}^1\to f_*\Omega^1_{X,d}\to \sD\to 0
%\end{equation} is basically Sequence (\ref{seqK}) restricted to the closed holomorphic one-forms.
%\end{proof}

\begin{lem}
\label{loc}
The sheaf $\sD$ is a local system.
\end{lem}
\begin{proof}
See \cite[Lemma 4.2]{PT} and \cite[Lemma 2.6]{RZ4} for the case where $Y$ is a curve.

 Call $j\colon Y^0 \to Y$ the inclusion of the locus of regular values of $f$. Since $\sD$ is a subsheaf of $R^1f_*\mathbb C$ by Diagram (\ref{dia1}), its restriction $j^*\sD$ is a subsheaf of the local system $j^*R^1f_*\mathbb C$, and $j^*\sD$ is itself a local system. 

We now check the monodromies of the local system $j^*\sD$ around the branches of $E$: we will see that they are trivial and therefore prove that $j_*j^*\sD$ is a local system.
Since by definition $\sD$ is the kernel of the morphism $R^1f_*\mC_X\to R^1f_*\sO_X$, its stalk over a point $y\in Y^0$ is contained in the kernel of the projection map $H^1(F_y,\mC)\to H^{0,1}(F_y)$ and so it is a vector subspace of $H^{1,0}(F_y)$. Here the standard Hermitian form is positive definite, hence the monodromy representation associated to $j^*\sD$ is unitary flat, furthermore it is unipotent for $f$ semistable as in our assumption, hence it is trivial. This proves that $j_*j^*\sD$ is a local system.

The last step consists in noticing that $  R^1f_*\mC_X\to j_*j^*R^1f_*\mC_X$ is surjective by the local invariant cycle theorem, see for example \cite[Theorem 1.4.1]{Cat}, so the natural map $\sD\to j_*j^*\sD$ is also surjective. Since it is an isomorphism on $Y^0$, it immediately follows that $\sD\to j_*j^*\sD$ is an isomorphism on $Y$ because otherwise the kernel would be a torsion subsheaf of $\sD$, and hence trivial since $\sD\hookrightarrow K_\partial$ and $K_\partial$ is torsion free.
%
%The last step consists in noticing that $  R^1f_*\mC_X\to j_*j^*R^1f_*\mC_X$ is an isomorphism by \cite[Lemma C.13]{PS} and \cite[Theorem 5.3.4]{CEZGT}, so $j_*j^*D\cong D$.
\end{proof}

We denote by $\mD$ the local system $\sD$ and we give another useful interpretation of $\mD$.
Note that the crucial point is that $\mD$ is defined as the image of $f_*\Omega^1_{X,d}\to R^1f_*\mC$.
Now if we consider the natural composition 
\begin{equation}
\label{composizione}
\Omega^1_{X,d}\to \Omega^1_{X}\to \Omega^1_{X/Y}
\end{equation}
and we take its pushforward 
\begin{equation}
f_*\Omega^1_{X,d}\to f_*\Omega^1_{X}\to f_*\Omega^1_{X/Y}
\end{equation} we see that the kernel of this composition is exactly $\Omega^1_{Y,d}$ and therefore $\mD$ can also be seen as a subsheaf of $f_*\Omega^1_{X/Y}$.
% by the sequence defining $\mD$
%\begin{equation}
%\label{seqdef}
%0\to \Omega_{Y,d}^1\to f_*\Omega^1_{X,d}\to \sD\to 0.
%\end{equation}

More precisely, consider 
$$
d_{X/Y}\colon \Omega^1_{X/Y}\to \Omega^2_{X/Y}
$$ the relative de Rham differential, that is the differential along the fibers induced by the usual de Rham differential $d\colon \Omega^1_{X}\to \Omega^2_{X}.$ Denote by $\Omega^1_{X/Y,d_{X/Y}}$ the kernel of $d_{X/Y}$, the sheaf of closed relative differential forms.

Now by the fact that by definition $\mD$ comes from closed differential forms on $X$ we immediately have that not only $\mD\subset f_*\Omega^1_{X/Y}$ but more specifically $\mD\subset f_*\Omega^1_{X/Y,d_{X/Y}}$.

\begin{rmk}We have that $j^*\mD$ is the largest local subsystem of $j^*R^1f_*\mC$ with stalk a subspace of $H^{1,0}(F_y)$ on the general fiber. Indeed, every other subsystem with the same property is contained in the kernel of $R^1f_*\mC_X\to R^1f_*\sO_X$ and so it is contained in $\mD$.
When the relative dimension $d=n-m$ equals 1, this means that $\mD$ is the local system which gives the second Fujita decomposition of $f_*\omega_{X/Y}$, see \cite{CK}. 
\end{rmk}

\subsection{Subsheaves of $f_*\Omega^{d}_{X/Y}$}

At the level of top forms on the fibers we can proceed in a similar way as follows.
As in Sequence (\ref{composizione}) we naturally have 
\begin{equation}
\Omega^d_{X,d}\to \Omega^d_X\to \Omega^d_{X/Y}
\end{equation} and by taking the pushforward 
\begin{equation}
\label{composizione2}
f_*\Omega^d_{X,d}\to f_*\Omega^d_X\to f_*\Omega^d_{X/Y}
\end{equation} we give the following definition.
\begin{defn}
We denote by $\sD^{d}$ the image of the map $f_*\Omega^d_{X,d}\to f_*\Omega^d_{X/Y}$ given by the composition in (\ref{composizione2}).
\end{defn}

Similarly to $\mD$, $\sD^{d}$ can be interpreted as a sheaf of top forms on the fibers which can be lifted to \textit{closed} holomorphic forms on  $X$.
We now show that $\sD^d$ is also a local system, more precisely:
we prove that $j^*\sD^d$ is a local system that can be extended to a local system on the whole base $Y$ and this extension is $\sD^d$. This is very similar to the case of $\mD$ of Lemma \ref{loc}.
\begin{lem}
\label{lem2}
The sheaf $j^*\sD^{d}$ is a local system that trivially extends on $Y$, hence $\sD^{d}$ is a local system on $Y$ which we will denote by $\mD^{d}$.
\end{lem}
\begin{proof}
	The sheaf $j^*\sD^{d}$ is the largest local system on $Y^0$ contained in $f_*\Omega^{d}_{X/Y}$ and the proof of this fact is the same as in the case over 1-dimensional base, cf. \cite[Lemma 3.4]{RZ4}.
	
The intersection form on $j^*\sD^{d}$ is, up to constant, strictly positive definite,  hence the
monodromy representation associated to $j^*\sD^{d}$ is unitary flat and unipotent by our semistability assumption, hence it is trivial. This proves that $j_*j^*\sD^{d}$ is a local system.
Exactly as we have seen in Lemma \ref{loc}, $ R^{d}f_*\mC\to j_*j^*R^{d}f_*\mC$ is surjective and we easily deduce the isomorphism $j_*j^*\sD^{d}\cong \sD^{d}$.
\end{proof}

By Lemma (\ref{lem2}) we get that $\mD^{d}$ is the largest local system contained in $f_*\Omega^{d}_{X/Y}$. By (\ref{injection}) we have the inclusion $f_*\Omega^{d}_{X/Y}\hookrightarrow f_*\omega_{X/Y}$ and we can conclude that

 %(see (\ref{inclusione})) which is an isomorphism on $B^0$. It immediately follows that 
\begin{thm}
\label{secfujita}
$\mD^{d}$ is the local system that gives the second Fujita decomposition of $f_*\omega_{X/Y}$, that is 
\begin{equation}
\label{fujitaii}
f_*\omega_{X/Y}=\sU\oplus \sW
\end{equation} with $\sW$ generically ample and $\sU=\mD^{d}\otimes \sO_Y$.
\end{thm}

\section{Massey products and local systems}
\label{sezioneaggiunta}

We recall that the notion of Massey products, originally called adjoint forms, for an infinitesimal deformation of an algebraic variety $X$ has been introduced in \cite{CP}, \cite{PZ}, \cite{RZ1}. Here we want to extend this construction on our fibration $f\colon X\to Y$ in such a way that for the general base point $y\in Y$ we obtain the classical construction on the fiber $F_y$; see \cite{PT} and \cite{RZ4} for the case where the base is a curve.
We stress that since the base $Y$ is of arbitrary dimension we need to properly generalize the theory that is built over a one dimensional base.

We now show how to construct the \emph{Massey product} of $d+1$ sections of $K_\partial$.

\subsection{Massey product of $1$-forms}
\label{sezaggiunte}
Given $d+1$ linearly independent sections $\eta_1,\ldots,\eta_{d+1}$ in $\Gamma(A,K_{\partial})$ on a suitable open subset $A\subset Y$, by Lemma \ref{split} we can always find liftings of the wedges $\eta_1\wedge\dots\wedge\widehat{\eta_i}\wedge \dots\wedge\eta_{d+1}$, $i=1,\ldots,{d+1}$, in $\Gamma(A, \bigwedge^{d} f_*\Omega^1_X)$
%\begin{equation}
%	%\label{wedge}
%	\bigwedge^{d} f_*\Omega^1_X\to \bigwedge^{d} K_{\partial}
%\end{equation} 
and then we apply the natural wedge map
\begin{equation}
\bigwedge^{d}f_*\Omega^1_X\to f_*\bigwedge^{d}\Omega^1_X=f_*\Omega^{d}_X
\end{equation}

\begin{defn}
	\label{omegai}
	We call $\omega_i\in \Gamma(A,f_*\Omega^{d}_X)$, $i=1,\ldots,d+1$, the sections corresponding to $\eta_1\wedge\dots\wedge\widehat{\eta_i}\wedge \dots\wedge\eta_{d+1}$ via this construction and $\sW$ the submodule of $\Gamma(A,f_*\Omega^{d+1}_X)$ given by $\langle\omega_i\rangle\otimes \Omega^1_Y$.
\end{defn}We assume that the $\omega_i$ are not all zero when restricted to the general fiber. 

Similarly we lift $\eta_1\wedge\dots\wedge\eta_{d+1}$ in $\Gamma(A, \bigwedge^{d+1} f_*\Omega^1_X)$
% in the map
%\begin{equation}
%%\label{split2}
%\bigwedge^{d+1} f_*\Omega^1_X\to \bigwedge^{d+1} K_{\partial}.
%\end{equation} 
and compose with
\begin{equation}
\bigwedge^{d+1}f_*\Omega^1_X\to f_*\bigwedge^{d+1}\Omega^1_X=f_*\Omega^{d+1}_X
\end{equation}
to obtain an element of  $\Gamma(A,f_*\Omega^{d+1}_X)$.
%\begin{equation}
%\label{agg}
%\bigwedge^{d+1} K_{\partial}\to f_*\Omega^{d+1}_X
%\end{equation}

Considering also $\Omega^2_Y\otimes f_*\Omega^{d-1}_X$ as a submodule of $f_*\Omega^{d+1}_X$ we can give the main definition
\begin{defn}
\label{mtrivial}
The \emph{Massey product} of $\eta_1,\ldots,\eta_{d+1}$ is the section in $\Gamma(A,f_*\Omega^{d+1}_X)$ computed from $\eta_1\wedge\dots\wedge\eta_{d+1}$ as described above. We say that the sections $\eta_1,\ldots,\eta_{d+1}$  are \emph{Massey trivial} if their Massey product is contained in the submodule $\sW+\Omega^2_Y\otimes f_*\Omega^{d-1}_X$.
\end{defn}

%\begin{rmk}
%Note that we do not ask that $A$ is an open contractible subset of $Y$ because by Lemma (\ref{split}) this construction can be done for arbitrary sections of $K_\partial$.
%\end{rmk}
\begin{rmk}While Massey products depend on the choice of the liftings in Sequence (\ref{seqK}),
it is easy to see that the definition of Massey triviality does not: two different choices differ by 1-forms in $\Omega^1_Y$. 
\end{rmk}
Now consider $s_1,\dots,s_{d+1}$ a choice of liftings of $\eta_1,\ldots,\eta_{d+1}$ via the splitting of Sequence (\ref{seqK}). A key result is the following, see also \cite[Proposition 4.5]{RZ4}.
\begin{prop}
\label{aggzero1}
If $\eta_1,\ldots,\eta_{d+1}$ are Massey trivial we can choose local liftings $\widetilde{s}_1,\dots,\widetilde{s}_{d+1}$ of the $\eta_i$'s such that $\widetilde{s}_1\wedge\dots\wedge\widetilde{s}_{d+1}=0$.
\end{prop}
\begin{proof}
Note that the Massey product of $\eta_1,\ldots,\eta_{d+1}$ is $s_1\wedge\dots\wedge s_{d+1}$ and the $\omega_i$ are exactly $\omega_i=s_1\wedge\dots\wedge\widehat{s_i}\wedge \dots\wedge s_{d+1}$.
Hence if $\eta_1,\ldots,\eta_{d+1}$ are Massey trivial then by definition we have that locally
\begin{equation}
\label{aggz1}
s_1\wedge\dots\wedge s_{d+1}=\sum_i \sigma_i\wedge s_1\wedge\dots\wedge\widehat{s_i}\wedge \dots\wedge s_{d+1}+\tau
\end{equation} with $\sigma_i$ 1-forms on $Y$ and $\tau$ a section of $\Omega^2_Y\otimes f_*\Omega^{d-1}_X$.
Defining 
\begin{equation}
\widetilde{s}_i=s_i+(-1)^i\sigma_i
\end{equation} we have that $\widetilde{s}_1\wedge\dots\wedge\widetilde{s}_{d+1}$ is in $\Omega^2_Y\otimes f_*\Omega^{d-1}_X$. Hence the following lemma will conclude the proof. 
\end{proof}
\begin{lem}
\label{tecnico}
If $\widetilde{s}_1,\dots,\widetilde{s}_{d+1}$ are such that $\widetilde{s}_1\wedge\dots\wedge\widetilde{s}_{d+1}$ is an element of $\Omega^2_Y\otimes f_*\Omega^{d-1}_X$ then $\widetilde{s}_1\wedge\dots\wedge\widetilde{s}_{d+1}=0$.
\end{lem}
\begin{proof}
This comes from the fact that the $\widetilde{s}_i$ are lifting of sections on the fibers and that at least one of the $\omega_i$ is non zero. Hence the claim can be verified as follows.
%Since not all the $\omega_i$ are zero, without loss of generality we assume that $\omega_{d+1}$ is not zero.
%Then $\widetilde{s}_1\wedge\dots\wedge\widetilde{s}_{d}$ restricts to $\omega_{d+1}$ and therefore $\widetilde{s}_1\wedge\dots\wedge\widetilde{s}_{d}$ is not an element of $\Omega^1_Y\otimes f_*\Omega^{d}_X$. Hence take the wedge with $\widetilde{s}_{d+1}$ and 
Consider the expression of $\widetilde{s}_1\wedge\dots\wedge\widetilde{s}_{d+1}$ in local coordinates on a suitable open subset of $f^{-1}(A)$. This is a $d+1\times n$ matrix. The hypothesis that $\widetilde{s}_1\wedge\dots\wedge\widetilde{s}_{d+1}$ is an element of $\Omega^2_Y\otimes f_*\Omega^{d-1}_X$ simply means that all the $d+1\times d+1$ minors of this matrix corresponding to the pieces in $\Omega^1_Y\otimes f_*\Omega^{d}_X$ have zero determinant and at this point it is not difficult to see that the determinant of all the other minors vanishes as well.
\end{proof}

\begin{rmk}
	\label{eccoperche}
	Proposition \ref{aggzero1} is actually a characterization of Massey triviality. This will be useful for the generalizations contained in the last section of this paper.
\end{rmk}

An alternative way to locally check Massey triviality on an $A\subset Y$ contractible open subset is the following.

Take local coordinates $y_1,\dots, y_m$ on $A$. For simplicity we denote by $dy$ the wedge product $dy_1\wedge\dots \wedge dy_m$ and by $\widehat{dy_k}$ the wedge of all the $dy_i$ except $dy_k$ so that $\widehat{dy_k}$ is a $m-1=n-d-1$-form. In the same way $\frac{\partial}{\partial y}$ will denote $\frac{\partial}{\partial y_1}\wedge\dots \wedge \frac{\partial}{\partial y_m}$.

Now consider again $s_1,\dots,s_{d+1}$ a choice of liftings of $\eta_1,\ldots,\eta_{d+1}$. The wedge product $s_1\wedge\dots\wedge s_{d+1}\wedge \widehat{dy_k}$ is a top form on $X$, an element of $\Gamma(A,f_*\omega_X)$ to be precise. Taking the contraction with $\frac{\partial}{\partial y}$ we obtain an element of $f_*\omega_{X/Y}$ by the isomorphism $f_*\omega_{X/Y}=f_*\omega_X\otimes \omega_Y^\vee$. The idea is to compare the sections constructed in this way with the $\eta_1\wedge\dots\wedge\widehat{\eta_i}\wedge \dots\wedge\eta_{d+1}$ seen as elements of $f_*\omega_{X/Y}$ via the inclusion (\ref{injection}).

\begin{prop}
\label{masslocale}
The sections $\eta_1,\ldots,\eta_{d+1}$ are Massey trivial if and only if, for every $k$, $s_1\wedge\dots\wedge s_{d+1}\wedge \widehat{dy_k}|_{\frac{\partial}{\partial y}}$ is an element in the submodule of $f_*\omega_{X/Y}$ generated by $\eta_1\wedge\dots\wedge\widehat{\eta_i}\wedge \dots\wedge\eta_{d+1}$.
\end{prop}
\begin{proof}
As we have seen in Proposition \ref{aggzero1}, if $\eta_1,\ldots,\eta_{d+1}$ are Massey trivial then by definition we have that 
\begin{equation}
s_1\wedge\dots\wedge s_{d+1}=\sum_i \sigma_i\wedge s_1\wedge\dots\wedge\widehat{s_i}\wedge \dots\wedge s_{d+1}+\tau
\end{equation} with $\sigma_i$ 1-forms on $Y$ and $\tau$ section of $\Omega^2_Y\otimes f_*\Omega^{d-1}_X$.
Taking the wedge with $\widehat{dy_k}$ followed by contraction with $\frac{\partial}{\partial y}$ immediately gives the first implication.

Vice versa assume that for every $k$ 
\begin{equation}
s_1\wedge\dots\wedge s_{d+1}\wedge \widehat{dy_k}|_{\frac{\partial}{\partial y}}=\sum_i f^k_i \eta_1\wedge\dots\wedge\widehat{\eta_i}\wedge \dots\wedge\eta_{d+1}
\end{equation} with $f^k_i$ local holomorphic functions. We consider this identity at the level of top forms on $X$ via the isomorphism $f_*\omega_{X/Y}=f_*\omega_X\otimes \omega_Y^\vee$, that is we wedge with $dy$,
\begin{equation}
s_1\wedge\dots\wedge s_{d+1}\wedge \widehat{dy_k}=\sum_i f^k_i s_1\wedge\dots\wedge\widehat{s_i}\wedge \dots\wedge s_{d+1}\wedge dy
\end{equation} which means that 
\begin{equation}
(s_1\wedge\dots\wedge s_{d+1}-\sum_i f^k_i s_1\wedge\dots\wedge\widehat{s_i}\wedge \dots\wedge s_{d+1}\wedge dy_{k})\wedge \widehat{dy_k}=0
\end{equation} 
Hence defining 
\begin{equation}
\sigma_i=\sum_k (-1)^{k-1}f^k_i dy_{k}
\end{equation} we get 
\begin{equation}
(s_1\wedge\dots\wedge s_{d+1}-\sum_i \sigma_i\wedge s_1\wedge\dots\wedge\widehat{s_i}\wedge \dots\wedge s_{d+1})\wedge \widehat{dy_k}=0
\end{equation} for every choice of $k$. 
This exactly means that 
\begin{equation}
s_1\wedge\dots\wedge s_{d+1}-\sum_i \sigma_i\wedge s_1\wedge\dots\wedge\widehat{s_i}\wedge \dots\wedge s_{d+1}=\tau
\end{equation} for a certain $\tau$ section of $\Omega^2_Y\otimes f_*\Omega^{d-1}_X$.
\end{proof}

%Alternatively we can define Massey products pointwise: given a regular value $b\in B$, we can take the 1-forms defined by $\eta_1,\dots,\eta_n$ on the fiber $F_b$ and build their adjoint form in the usual way, which we briefly recall. For an extensive discussion on the topic, we refer to \cite{RZ1}.
%
%Call $W<H^0(F_b,\Omega^1_{F_b})$ the $n$-dimensional subspace generated by these sections and take $s_1,\dots,s_n\in H^0(F_b,\Omega^1_{X|F_b})$ a choice of liftings on $X$. These lifting can be found since $W\subset \ker \delta_{\xi_b}$ by (\ref{kerxi}).
%
%We have a top form $\Omega\in H^{0}(F_b,\omega_{X|F_b})$ from the 
%element $s_{1}\wedge\ldots\wedge s_{n}\in 
%\bigwedge^{n+1}H^{0}(F_b,\Omega^1_{X|F_b})$. Since $\omega_{X|F_b}\cong \omega_{F_b}$ we 
%actually obtain from $\Omega$ a top form $\omega$ of $\omega_{F_b}$. Such an $\omega\in 
%H^{0}(X,\omega_{F_b})$ is the classical \emph{adjoint form} or \emph{Massey product}. When needed we will denote it by $m_{\xi_b}(\eta_1,\ldots,\eta_n)$.
%
 %All the pointwise defined adjoint forms can be glued together to an element of $f_*\omega_{X/B}$ modulo the $\sO_B$-submodule $\sW$. This construction agrees with the previous one of Definition \ref{mtrivial} on suitable open subsets $A\subset B$.

%, but the first is more general and is more convenient to take care of the critical values of the map $f$.

Of course since $\mD$ is a subsheaf of $K_{\partial}$, it makes sense to construct Massey products starting from sections $\eta_i\in \Gamma(A,\mD)$.
Note that in this case the liftings of the $\eta_i$ are closed holomorphic forms on $X$, hence the $\omega_i$ are also closed since they are wedge of closed forms. This means that their restriction to the fibers is in  $\mD^{d}$.
% i.e. consider the map 
%\begin{equation}
%\bigwedge^{d+1} \mD \to f_*\Omega_{X}^{d+1}
%\end{equation} instead of (\ref{agg}).
%We now have the following diagram
%\begin{center}
%\begin{equation}
%\xymatrix{
%&\bigwedge^{n-1} K_{\partial}\ar[dr]&\\
%\bigwedge^{n-1} \mD\ar@{^{(}->}[ur]\ar[dr]\ar[r]&\mD^{n-1}\ar@{^{(}->}[d]&f_*\Omega^{n-1}_{X/B}\ar@{^{(}->}[r]&f_*\omega_{X/B}\\
%&f_*\Omega^{n-1}_{X/B,d_{X/B}}\ar@{^{(}->}[ur]&
%}
%\end{equation}
%\end{center}
%\begin{center}
%\begin{equation}
%\xymatrix{
%&\bigwedge^{n-1} \mD\ar[d]&&&\\
%f_*\Omega_{X,d}^{n-1}\ar@{->>}[r]&\mD^{n-1}\ar@{^{(}->}[r]&f_*\Omega^{n-1}_{X/B,d_{X/B}}\ar@{^{(}->}[r]&f_*\Omega^{n-1}_{X/B}\ar@{^{(}->}[r]&f_*\omega_{X/B}
%}
%\end{equation}
%\end{center}

%On the other hand note that the map (\ref{wedge}) restricted on $\mD$ has image in $\mD^{d}$, that is
%\begin{equation}
%\bigwedge^{d}\mD\to \mD^{d}
%\end{equation} because the wedge product of forms which can be lifted to closed holomorphic forms can also be lifted to a closed holomorphic form.

\subsection{Massey triviality of a vector space and strictness}
Let $A\subset Y$ be a contractible open subset and $W\subset \Gamma(A, K_{\partial})$ a vector subspace of dimension at least $d+1$.

We give the following definition
\begin{defn}
\label{mastriv}
We say that $W$ is Massey trivial if any $d+1$-uple of linearly independent sections in $W$ is Massey trivial by Definition \ref{mtrivial}.
\end{defn}

\begin{defn}
\label{strict}
We say that $W$ is strict if the map 
$$
\bigwedge^{d}W\otimes \sO_A\to f_*\omega_{X/Y_{|A}}
$$ is an injection of sheaves.
\end{defn} See \cite[Definition 4.5]{RZ4}

\begin{rmk}\label{remstrict}
%Given an $n$-dimensional variety $Y$, a subspace $W$ of $H^0(Y,\Omega^1_Y)$ is usually called strict if the natural map from $\bigwedge^n W$ to $H^0(Y,\omega_Y)$ is an isomorphism on the image. See \cite[Definition 2.1 and 2.2]{Ca2}, \cite[Definition 2.2.1]{RZ1} and \cite[Definition 4.4]{RZ4}.
%
%If $W$ as in Definition \ref{strict} is a subspace of sections of $\mD$, $W\subset\Gamma(A,\mD)\subset\Gamma(A,K_\partial)$, then we can see $W$ as a subspace of $H^0(F_y,\Omega^1_{F_y})$ since $\mD$ is a local system. If $W$ is strict in the usual sense then it is strict according to Definition (\ref{strict}). In fact if $W$ is strict in the usual sense we have the injection $\bigwedge^{n-1}W\hookrightarrow \mD^{n-1}$ and taking the tensor product by $\sO_A$ we immediately get the desired injection $$\bigwedge^{n-1}W\otimes \sO_A\hookrightarrow \mD^{n-1}\otimes \sO_A\hookrightarrow f_*\omega_{X/Y_{|A}}.$$

Note that if $d=1$ and $W$ is a subspace of sections of the local system $\mD$ then $W$ is always strict.
\end{rmk}

\begin{prop}
\label{wedge0}
Let $A\subset Y$ be a contractible open subset and $W\subset \Gamma(A, K_{\partial})$ (resp. $W\subset \Gamma(A, \mD$)) a Massey trivial strict subspace. Then there exists a unique $\widetilde{W}\subset \Gamma(A,f_*\Omega^1_X)$ (resp. $\widetilde{W}\subset \Gamma(A,f_*\Omega^1_{X,d})$ ) lifting $W$ such that the wedge map
$$
\psi\colon \bigwedge^{d+1}\widetilde{W} \to  \Gamma(A,f_*\Omega_{X}^{d+1})
$$ is zero.
\end{prop}
\begin{proof}
We work by induction on the dimension of $W$ and note that for $\dim W=d+1$ this is Proposition \ref{aggzero1}.

%If $\dim W=n$ take $\eta_1,\ldots,\eta_{n}$ a basis of $W$ and $s_1,\dots,s_{n}\in \Gamma(A,f_*\Omega^1_S)$ arbitrary liftings via Lemma (\ref{split}). 
%
%By hypothesis there exist $a_i$ holomorphic functions on $A$ such that
%\begin{equation}
%\label{ipotesi}
%\psi(s_1\wedge\dots\wedge s_{n}\otimes \frac{\partial}{\partial t})=\sum^{n}_{i=1} a_i\omega_i
%\end{equation} where the $\omega_i$ are as in Definition (\ref{omegai}).
%
%Defining a new lifting for the element $\eta_i$:
%\begin{equation*}
%t_i:=s_i+(-1)^{n-i}a_i\cdot dt
%\end{equation*} we have 
%\begin{equation*}
%\begin{split}
%\psi(t_1\wedge\dots\wedge t_{n}\otimes \frac{\partial}{\partial t})=\psi(s_1\wedge\dots\wedge s_{n}\otimes \frac{\partial}{\partial t})-\sum_{i=1}^{n} a_i\psi(s_1\wedge\dots\wedge\widehat{s_i}\wedge\dots\wedge s_{n}\wedge dt\otimes \frac{\partial}{\partial t})=\\
%\psi(s_1\wedge\dots\wedge s_{n}\otimes \frac{\partial}{\partial t})-\sum^{n}_{i=1} a_i\omega_i=0.
%\end{split}
%\end{equation*}
%
%No we perform the induction step. 
So assume now that $\dim W=l>d+1$.

We take a basis $\eta_1,\dots,\eta_l$ of $W$ and by induction we can choose liftings $t_1,\dots,t_{l-1}$ of $\eta_1,\dots,\eta_{l-1}$ such that the map $\psi$ is zero on $\langle t_1,\dots,t_{l-1} \rangle$. Finally we take $s_l$ a lifting of $\eta_l$.

%Calling $\phi$ the map in (\ref{wedge}), 

By the hypothesis of Massey triviality for every $d+1$ elements in $W$ and by the interpretation provided by  Proposition \ref{masslocale}, for every $k$ there exists holomorphic $b^k,b'^k$ and $a^k_i, c^k_i$, $i=1,...,d$, such that 
\begin{equation*}
t_1\wedge\dots\wedge t_{d}\wedge s_l\wedge \widehat{dy_k}|_{\frac{\partial}{\partial y}}=\sum^{d}_{i=1}a^k_i \eta_1\wedge\dots\wedge\widehat{\eta_i}\wedge\dots\wedge \eta_{d}\wedge\eta_l+b^k \eta_1\wedge\dots\wedge \eta_{d}
\end{equation*} (this is by Massey triviality of $\eta_1,\dots,\eta_d,\eta_l$) and 
\begin{equation*}
\begin{split}
t_1\wedge\dots\wedge t_{d}\wedge (\sum_{i=d+1}^{l-1}t_i+s_l)\wedge \widehat{dy_k}|_{\frac{\partial}{\partial y}}=\sum^{d}_{i=1}c^k_i \eta_1\wedge\dots\wedge\widehat{\eta_i}\wedge\dots\wedge \eta_{d}\wedge (\sum_{i=d+1}^{l-1}\eta_i+\eta_l)+b'^k \eta_1\wedge\dots\wedge \eta_{d}
\end{split}
\end{equation*} (this is by Massey triviality of $\eta_1,\dots,\eta_d,\sum_{i=d+1}^{l-1}\eta_i+\eta_l$).

Now by the induction hypothesis we have that 
\begin{equation*}
t_1\wedge\dots\wedge t_{d}\wedge s_l\wedge \widehat{dy_k}|_{\frac{\partial}{\partial y}}=
t_1\wedge\dots\wedge t_{d}\wedge (\sum_{i=d+1}^{l-1}t_i+s_l)\wedge \widehat{dy_k}|_{\frac{\partial}{\partial y}}
\end{equation*}
and by the strictness of $W$ we can compare their expressions and obtain $b^k=b'^k$ and $a^k_i=c^k_i=0$.

Therefore 
\begin{equation*}
t_1\wedge\dots\wedge t_{d}\wedge s_l\wedge \widehat{dy_k}|_{\frac{\partial}{\partial y}}=b^k \eta_1\wedge\dots\wedge \eta_{d}
\end{equation*} and now choosing $t_l=s_l-\sum_k b^k dy_k$ we have that 
\begin{equation*}
t_1\wedge\dots\wedge t_{d}\wedge t_l\wedge \widehat{dy_k}|_{\frac{\partial}{\partial y}}=0
\end{equation*} for every $k$, hence $t_1\wedge\dots\wedge t_{d}\wedge t_l$ is also zero as a section of $f_*\Omega_X^{d+1}$ by Lemma \ref{tecnico}.

It remains to prove that given any choice of indices $\{j_1,\dots,j_{d}\}\subset \{1,\dots,l\}$ we have 
\begin{equation*}
t_{j_1}\wedge\dots\wedge t_{j_{d}}\wedge t_l=0.
\end{equation*}

It is enough to show this claim for $t_s\wedge t_2\wedge \dots\wedge t_{d}\wedge t_l$ where $s\neq 1$ since the same procedure can be iterated if necessary to obtain the general statement.

Using $t_1\wedge\dots\wedge t_{d}\wedge t_l=0$ and the induction hypothesis we have the equalities
\begin{equation}
t_{s}\wedge t_2\wedge\dots\wedge t_{{d}}\wedge (t_l+t_1)=t_{s}\wedge t_2\wedge\dots\wedge t_{{d}}\wedge t_l=(t_{s}+t_1)\wedge t_2\wedge\dots\wedge t_{d}\wedge t_l
\end{equation}
Taking the $\wedge \widehat{dy_k}|_{\frac{\partial}{\partial y}}$ on the first equality and using the Massey triviality we obtain
\begin{multline*}
\alpha^k_s\eta_2\wedge\dots\wedge\eta_{{d}}\wedge(\eta_l+\eta_1)+\sum^{d}_{j=2}\alpha^k_j\eta_{s}\wedge\eta_2\dots\wedge\widehat{\eta_{j}}\wedge\dots\wedge\eta_{{d}}\wedge(\eta_l+\eta_1)+\alpha^k_l \eta_{s}\wedge\eta_2\wedge\dots\wedge\eta_{{d}}=\\
=\beta^k_s\eta_2\wedge\dots\wedge\eta_{{d}}\wedge\eta_l+\sum^{d}_{j=2}\beta^k_j\eta_{s}\wedge\eta_2\dots\wedge\widehat{\eta_{j}}\wedge\dots\wedge\eta_{{d}}\wedge\eta_l+\beta^k_l\eta_{s}\wedge\eta_2\wedge\dots\wedge\eta_{{d}}
\end{multline*}
Doing the same for the second equality we have

\begin{multline*}
\beta^k_s\eta_2\wedge\dots\wedge\eta_{{d}}\wedge\eta_l+\sum^{d}_{j=2}\beta^k_j\eta_{s}\wedge\eta_2\dots\wedge\widehat{\eta_{j}}\wedge\dots\wedge\eta_{{d}}\wedge\eta_l+\beta^k_l\eta_{s}\wedge\eta_2\wedge\dots\wedge\eta_{{d}}=\\
=\mu^k_s\eta_2\wedge\dots\wedge\eta_{{d}}\wedge\eta_l+\sum^{d}_{j=2}\mu^k_j(\eta_{s}+\eta_1)\wedge\eta_2\dots\wedge\widehat{\eta_{j}}\wedge\dots\wedge\eta_{{d}}\wedge\eta_l+\mu^k_l (\eta_{s}+\eta_1)\wedge\eta_2\wedge\dots\wedge\eta_{{d}}
\end{multline*}
Where the $\alpha^k$, $\beta^k$ and $\mu^k$ are holomorphic functions on $A$.
Using the strictness as before we immediately find that all these functions vanish, for every $k$, giving us the desired result, again via Lemma \ref{tecnico}.

In the case $W\subset \Gamma(A,\mD)\subset\Gamma(A,K_\partial)$, since by definition of $\mD$ all the liftings of the sections $\eta_i$ are closed forms, we have that $\widetilde{W}\subset \Gamma(A,f_*\Omega^1_{S,d})$.
\end{proof}

We finally relate in some sense the property of being locally Massey trivial and globally Massey trivial.

\begin{prop}
\label{lglob}
Let $W\subset \Gamma(Y, K_{\partial})$ be a strict subspace of global sections of $K_{\partial}$ and let $A\subset Y$ an open contractible subset. If the sections of $W$ are Massey trivial when restricted to $A$ then they are Massey trivial everywhere. 
\end{prop}
\begin{proof}
Since the dimension of $W$ doesn't play any role, for simplicity we assume that $\dim W=d+1$. Take $\eta_1,\ldots,\eta_{d+1}$ a basis of $W$ and $\widetilde{W}=\langle s_1,\dots,s_{d+1}\rangle < \Gamma(Y,f_*\Omega^1_X)$ a lifting of $W$ via Lemma \ref{split}. The $s_i$'s are global holomorphic 1-forms on $X$. 

By hypothesis there exist $a^k_i$ holomorphic functions on $A$ such that
\begin{equation}
\label{ipotesi1}
s_1\wedge\dots\wedge s_{d+1}\wedge \widehat{dy_k}|_{\frac{\partial}{\partial y}}=\sum^{d+1}_{i=1} a^k_i\eta_1\wedge\dots\wedge\widehat{\eta_i}\wedge\dots\wedge\eta_{d+1}
\end{equation}
Now take $A'\subset Y$ another open contractible subset with nontrivial intersection with $A$ and call $y_i'$ the local coordinates on $A'$.
%$$
%\psi'\colon \widetilde{W}\otimes T_B\to f_*\omega_{X/Y_{|A'}}
%$$ the map which gives the Massey product on $A'$.
By the strictness hypothesis we can complete the $\eta_1\wedge\dots\wedge\widehat{\eta_i}\wedge\dots\wedge\eta_{d+1}$ to a local frame of $f_*\omega_{X/Y_{|A'}}$, hence in this frame we have
$$
s_1\wedge\dots\wedge s_{d+1}\wedge \widehat{dy'_k}|_{\frac{\partial}{\partial y'}}=\sum^{d+1}_{i=1} a'^k_i\eta_1\wedge\dots\wedge\widehat{\eta_i}\wedge\dots\wedge\eta_{d+1}+\sum_{j} b'^k_j\tau_j.
$$
By comparing this with (\ref{ipotesi1}) on $A\cap A'$ and using the relations $\partial/\partial y'=\partial y/\partial y'\cdot \partial/\partial y$ and $dy'=\partial y'/\partial y \cdot dy$
%we have 
%\begin{multline}
%\sum^{n}_{i=1} a'_i\omega_i+\sum_{j} b_j\tau_j=\psi'(s_1\wedge\dots\wedge s_{n}\otimes \frac{\partial}{\partial t'})=\\=\frac{\partial t}{\partial t'}\psi(s_1\wedge\dots\wedge s_{n}\otimes \frac{\partial}{\partial t'})=\sum^{n}_{i=1} \frac{\partial t}{\partial t'}a_i\omega_i.
%\end{multline}
we are immediately able to say that the $b'^k_j$ vanish on $A\cap A'$ and hence everywhere on $A'$. This proves the Massey triviality on $A'$.
By iterating on an appropriate cover of $Y$ we are done.  
\end{proof}
By the exact same computations it also follows that the $a^k_i$ and $a'^k_i$ transform as coefficients of a 1-form on $Y$, hence the forms $\sigma_i$ as in the proof of Proposition \ref{masslocale} define global 1-forms.
%hence and we can use them to define global 1-forms $\sigma_i$ such that  
%$$
%s_1\wedge\dots\wedge s_{d+1}=\sum_{i=1}^n \sigma_i\wedge s_1\wedge\dots\wedge\widehat{s_i}\wedge\dots\wedge s_{d+1}+\tau,
%$$ $\tau$ section of $\Omega^2_Y\otimes f_*\Omega^{d-1}_X$.
%Hence we have that 
%\begin{prop}
%\label{locglob}
%Under the hypotheses of the previous Proposition, call  $\widetilde{W}$ a lifting of $W$ in $H^0(Y,f_*\Omega^1_X)$. We have that the image of $\bigwedge^{d+1}\widetilde{W} \to H^0(Y,f_*\Omega_X^{d+1})$ is contained in the image of $$H^0(Y,\Omega^1_Y)\otimes \bigwedge^{d}\widetilde{W}+H^0(Y,\Omega^2_Y\otimes f_*\Omega_X^{d-1})\to H^0(Y,f_*\Omega_X^{d+1}).$$ In particular if $Y=\mP^m$ we have that the image of $\bigwedge^{n}\widetilde{W} \to H^0(Y,f_*\Omega_X^{d+1})$ is zero.
%\end{prop}
%\begin{proof}
%We take $\eta_i$ and $s_i$ as in the proof of the previous Proposition.
%By the same computations we get on $A\cap A'$ the relation 
%$$
%a_i'=\frac{\partial t}{\partial t'}a_i
%$$ hence the 1-forms on $B$ defined by $a_idt$ and $a_i'dt'$ glue together and produce global 1-forms on $B$, call them $\sigma_i$. From the Massey triviality we then obtain
%$$
%s_1\wedge\dots\wedge s_n=\sum_{i=1}^n \sigma_i\wedge s_1\wedge\dots\wedge\widehat{s_i}\wedge\dots\wedge s_n
%$$ which is our thesis.
%\end{proof}

These results allow us to give a \lq\lq local to global\rq\rq version of Proposition \ref{wedge0}.

\begin{prop}
\label{localglobal}
Let $W\subset \Gamma(Y, K_{\partial})$ a strict subspace of global sections of $K_{\partial}$ and let $A\subset Y$ be an open contractible subset. If the sections of $W$ are Massey trivial when restricted to $A$ then there exist a unique lifting $\widetilde{W}\subset \Gamma(Y, f_*\Omega^1_X)$ such that 
$$
\bigwedge^{d+1}\widetilde{W}\to \Gamma(Y,f_*\Omega_{X}^{d+1})
$$ is zero. If furthermore $W\subset \Gamma(Y, \mD)$ then $\widetilde{W}\subset \Gamma(Y, f_*\Omega^1_{X,d})$.
\end{prop}

\begin{rmk}
The last statement will be meaningful when $Y$ is not compact, otherwise global 1-forms on $X$ are closed.  
\end{rmk}
%
%TOGLIERE??????
%\begin{rmk}
%Given an $n$-dimensional variety $Y$, a subspace $W$ of $H^0(Y,\Omega^1_Y)$ is usually called strict if the natural map from $\bigwedge^n W$ to $H^0(Y,\omega_Y)$ is an isomorphism on the image. See \cite[Definition 2.1 and 2.2]{Ca2} and \cite[Definition 2.2.1]{RZ1}.
%
%If $W$ as above is a subspace of sections of $\mD$, $W\subset\Gamma(A,\mD)\subset\Gamma(A,K_\partial)$, then we can see $W$ as a subspace of $H^0(F_y,\Omega^1_{F_y})$ since $\mD$ is a local system. If $W$ is strict in the usual sense then it is strict according to Definition (\ref{strict}). In fact if $W$ is strict in the usual sense we have the injection $\bigwedge^{n-1}W\hookrightarrow \mD^{n-1}$ and taking the tensor product by $\sO_A$ we immediately get the desired injection $$\bigwedge^{n-1}W\otimes \sO_A\hookrightarrow \mD^{n-1}\otimes \sO_A\hookrightarrow f_*\omega_{X/Y_{|A}}.$$
%\end{rmk}

\section{A relative version of the Castelnuovo-de Franchis Theorem}
\label{sez4}

In \cite[Theorem 5.6]{RZ4} we have shown this slightly different version of the generalized Castelnuovo-de Franchis theorem, cf. \cite[Theorem 1.14]{Ca2} and \cite[Prop II.1]{Ran},

\begin{thm}
\label{cas2}
Let $X$ be a compact K\"ahler manifold  and $w_1,\dots, w_l \in H^0 (X,\Omega^1_X)$ linearly independent 1-forms such that $w_{j_1}\wedge\dots\wedge w_{j_{k+1}}= 0$ for every $j_1,\dots,j_{k+1}$ and that no collection of $k$ linearly independent forms in the span of $w_1,\dots, w_{j_{k+1}}$ wedges to zero. Then there exists a holomorphic map $f\colon X\to Z$ over a normal variety $Z$ of dimension $\dim Z=k$ and such that $w_i\in f^*H^0(Z,\Omega^1_Z)$. Furthermore $Z$ is of general type.
\end{thm}

Now the point is that Massey triviality gives a relative version of this theorem, that is a Castelnuovo-de Franchis theorem for our fibration $X\to Y$.  

Let $W<\Gamma(A, \mD)$ be a Massey trivial subspace. Let $H$ be the kernel of the monodromy representation of $\mD$, which is a normal subgroup of $\pi_1(Y, y)$, and call $H_W$ the subgroup acting trivially on $W$. For every subgroup $K<H_W$, we denote by $Y_K\to Y$ the \'{e}tale base change of group $K$ and by $X_K\to Y_K$ the associated fibration. We have the following:
\begin{thm}
\label{castmassey}
Let $X\to Y$ be a semistable fibration. Let $A\subset Y$ be an open subset and $W\subset \Gamma(A,\mD)$ a Massey trivial strict subspace.
Then $X_K$ has a morphism $h_K\colon X_K\to Z$ over a normal $(n-1)$-dimensional variety of general type $Z$ such that $W\subset h_K^*(H^0(Z,\Omega^1_Z))$. Furthermore if $W$ is a maximal Massey trivial strict subspace we have the equality $W= h_K^*(H^0(Z,\Omega^1_Z))$.
\end{thm}
\begin{proof}
	See \cite[Theorem 5.8]{RZ4} for the case where $Y$ is a curve. Here the proof is similar and we give an idea for $K=H_W$.
 
The key point is that since $H_W$ acts trivially on $W$, with the base change $Y_{H_W}\to Y$ the sections of $W$ are global sections on $Y_{H_W}$. More precisely call $\mD_W$ the local system on $Y_{H_W}$ obtained by inverse image of $\mD$, then the elements of $W$ extend to global sections in $\Gamma(Y_{H_W},\mD_W)$.
We can then apply Proposition \ref{localglobal} and find global 1-forms of $X_W$ which satisfy the hypotheses of the Castelnuovo-de Franchis Theorem \ref{cas2}. In particular note that the strictness hypothesis ensures that no collection of $k$ linearly independent forms in the span of these sections wedges to zero.
Note that even in the case $X_{H_W}$ not compact, the proof works because by Proposition \ref{localglobal} the liftings are closed 1-forms and the proof of the Castelnuovo-de Franchis also works if $X$ is not compact if we further assume that the 1-forms are closed.
%Castelnuovo-de Franchis still works as pointed out in Remark \ref{bastachiuse}.
\end{proof}

\section{The monodromy of $\mD$ and $\mD^d$}
\label{sez5}
In this section we study the monodromy of $\mD$, similarly to \cite{PT}, \cite{RZ4}.
We start by recalling that a subspace $W\subset \Gamma(A,\mD)$ naturally generates a local system in $\mD$ by taking the closure of $W$ under the monodromy action.
More precisely, denote by $\rho$ the monodromy map associated with $\mD$ and by $G=\pi_1(Y, y)/\ker \rho$ the monodromy group acting non-trivially on $\mD$. 
The local system generated by $W$ is by definition the local system with stalk $\widehat{W}=\sum_{g\in G}g\cdot W$. We will denote it by $\mW$.
\begin{defn}
\label{mastrivgen}
If $W$ is Massey trivial, we will say that $\mW$ is Massey trivial generated.
\end{defn} See \cite[Definition 5.5]{PT}, \cite[Definition 5.9]{RZ4}.

To the local system $\mW$ we associate its monodromy group as follows. The action of the fundamental group $\pi_1(Y, y)$ on the stalk of $\mD$ restricts to an action $\rho_\mW$ on the stalk of $\mW$, that is 
$$
\rho_\mW\colon \pi_1(Y, y)\to \text{Aut}(\widehat{W}).
$$

We now introduce a suitable set $\sK$ together with an action of the monodromy group $G_\mW=\pi_1(Y, y)/\ker \rho_\mW$ on $\sK$. Sometimes we will denote $\ker \rho_\mW$ by $H_\mW$, so that $G_\mW=\pi_1(Y,y)/H_\mW$.

%We will now construct an action of the monodromy group $G_\mW=\pi_1(B, b)/\ker \rho_\mW\cong \Ima \rho_\mW$ on a suitable set which will allow us to study this group. Sometimes we will denote $\ker \rho_\mW$ by $H_\mW$, so that $G_\mW=\pi_1(B,b)/H_\mW$.

Let $u_\mW\colon Y_\mW\to Y$ be the covering classified by the subgroup $H_\mW$ and $f_\mW\colon X_\mW\to Y_\mW$ the associated fibration. As before, the inverse image of the local system $\mW$ on $Y_\mW$ is trivial and we will often identify the sections of $W$ and their unique liftings provided by Proposition \ref{localglobal}, which are global closed 1-forms on $X_\mW$.

By Theorem \ref{castmassey} applied to the subgroup $H_\mW$ we get a map $h\colon X_{\mW}\to Z$ which can be composed with the action of $G_\mW$ on $X_\mW$ obtained from the standard action of $G_\mW$ on $Y_\mW$. We call $h_g$ the composition and we restrict it to $F_0$, the fiber of $f_\mW$ over a regular point $y_0$.
%
%\begin{equation}
%\xymatrix{
%X_\mW\ar^{g}[r]\ar_{h_g}[rrd]&X_\mW\ar^{h}[rd]&\\
%&&Z
%}
%\end{equation}
%Now we take a regular point $y\in A\subset Y$ and a preimage $y_0$ of $y$ via $Y_\mW\to Y$. Denote by  and consider the above diagram restricted to $F_0$.

\begin{equation}
\xymatrix{
F_0\ar@{^{(}->}[r]\ar_{k_g}[drrr]&X_\mW\ar^{g}[r]\ar^{h_g}[rrd]&X_\mW\ar^{h}[rd]&\\
&&&Z
}
\end{equation}

\begin{rmk}
The map $F_0\to Z$ is surjective since the pullbacks of the 1-forms of $Z$ give on $F_0$ the subspace $W$ which is strict by hypothesis.
\end{rmk}

We can define two sets of functions 
$$
\mathcal{H}=\{h_g\colon X_\mW\to Z\mid g\in G_\mW\}
$$ and
$$
\mathcal{K}=\{k_g\colon F_0\to Z\mid g\in G_\mW\}
$$ and the natural action of $G_\mW$ by $g_1\cdot h_{g_2}=h_{g_2g_1}$ and $g_1\cdot k_{g_2}=k_{g_2g_1}$. From now on we will assume that $W$ is Massey trivial and we will prove that the action $G_\mW\times \sK\to \sK$ is faithful.
We recall without proof the following lemma see \cite[Lemma 6.1]{PT}, \cite[Lemma 5.11]{RZ4}.
\begin{lem}
\label{formula1}
Let $e$ be the neutral element of $G_\mW$ and $\alpha\in H^0(Z,\Omega^1_Z)$. Then for each $g\in G_\mW$,
\begin{equation}
k_g^*(\alpha)=g^{-1}k_e^*(\alpha)
\label{formula}
\end{equation}
where $g^{-1}$ acts on $k_e^*(\alpha)\in W$ via the monodromy action defining $\mW$.
\end{lem}
%\begin{proof}
%Let $\beta=h_e^*(\alpha)$ the global closed 1-form in $X_\mW$ obtained by pullback. Clearly 
%$$
%k_g^*(\alpha)=(g^*h_e^*(\alpha))_{|F}=g^*\beta_{|F}=\beta_{|F_{g^{-1}{b}}}
%$$
%On the other hand
%$$
%g^{-1}k_e^*(\alpha)=g^{-1}\beta_{|F}=\beta_{|F_{g^{-1}{b}}}.
%$$ 
%\end{proof}
Thanks to this Lemma we can prove 
\begin{lem}
\label{faithful}
The action of $G_\mW$ on $\sK$ is faithful.
\end{lem}
\begin{proof}
Take $g\in G_\mW$, $g\neq e$ and we prove that there exist an element $k_{g'}$ of $\sK$ such that $g\cdot k_{g'}\neq k_{g'}$. Since by definition of the action we have $g\cdot k_{g'}=k_{g'g}$ we have to prove that 
$$
k_{g'}\colon F_0\to Z \quad\text{and}\quad k_{g'g}\colon F_0\to Z
$$ are different morphisms. We will prove this statement at the level of 1-forms, more precisely we prove that 
$$
k_{g'}^*\colon H^0(Z,\Omega^1_Z)\to H^0(F_0,\Omega^1_{F_0}) \quad\text{and}\quad k_{g'g}^*\colon H^0(Z,\Omega^1_Z)\to H^0(F_0,\Omega^1_{F_0})
$$ are different.

Now since $g\neq e$, there exist an element of $\mW$ which is not fixed by $g$, and, since $\mW$ is the local system generated by $W$, we can assume that this element is of the form $\hat{g}w$ with $w\in W$ and $\hat{g}\in G$, that is $g\hat{g}w\neq \hat{g}w$.

 By Theorem \ref{castmassey}, $w=k_e^*(\alpha)$ for some $\alpha \in H^0(Z,\Omega^1_Z)$, and by the previous Lemma we obtain 
$$
k_{g'}^*(\alpha)={g'}^{-1}k_e^*(\alpha)={g'}^{-1}w
$$ and
$$
k_{g'g}^*(\alpha)=({g'g})^{-1}k_e^*(\alpha)=({g'g})^{-1}w=g^{-1}g'^{-1}w
$$
Hence the thesis follows by taking $g'=\hat{g}^{-1}$.
\end{proof}
%ATTENZIONE PULLBACK INIETTIVO?
\subsection{Massey trivial generation and finiteness of monodromy groups}
We can now state the main theorem of this section:

\begin{thm}
\label{monfin}
Let $f \colon X \to Y$ be a semistable fibration and let $\mW\leq\mD$ be a local system generated by a maximal strict Massey trivial subspace.
Then the associated monodromy group $G_\mW$ is finite and the fiber of $\mW$ is isomorphic to 
$$
\sum_{g\in G_\mW} k_{g}^*H^0(Z,\Omega^1_Z).
$$
%Where $h\colon X_{\mW}\to Y$ is the higher irrational pencil  associated to ${\mW}$.
\end{thm}
\begin{proof}
By Lemma \ref{faithful}, we have an inclusion 
$$
G_\mW\hookrightarrow \text{Aut}(\sK)
$$ hence it is enough to show that $\sK$ is a finite set. $\sK$ is contained in $\text{Mor}(F_0,Z)$ the set of all surjective morphisms from $F_0$ to $Z$ and this is finite being $Z$ of general type, see for example \cite[Theorem 1]{kob}

For the result on the stalk of $\mW$, recall that this stalk is $\widehat{W}=\sum_{g\in G}g\cdot W$. Hence we have
$$
\widehat{W}=\sum_{g\in G}g\cdot W=\sum_{g\in G_\mW}g\cdot W=\sum_{g\in G_\mW} k_{g}^*H^0(Z,\Omega^1_Z)
$$where the second equality comes from the fact that $H_\mW$ fixes $W$ and the last from Lemma \ref{formula1}.
\end{proof}

Now if $\mD$ itself is Massey trivial generated we have the immediate corollary:
\begin{cor}
\label{fin2}
If $\mD$ is Massey trivial generated by a strict subspace, then the monodromy group $G$ is finite.
If furthermore the map $\bigwedge^{d}\mD\to \mD^{d}$ is surjective, the local system $\mD^{d}$ also has finite monodromy. 
\end{cor}
\begin{proof}
The first statement is immediate by Theorem \ref{monfin}. 

For the second it is enough to note that if $\bigwedge^{d}\mD\to \mD^{d}$ is a surjective map of local systems, then the monodromy group of $\mD^{d}$ is a subgroup of the monodromy group of $\bigwedge^{d}\mD$ and latter is finite by the first statement.
\end{proof}

We recall that $\mD^{d}$ is the local system associated to the second Fujita decomposition by Theorem \ref{fujitaii}. Furthermore the finiteness of the monodromy group of a local system is equivalent to the semi-ampleness of the unitary flat vector bundle, see for example \cite[Theorem 2.5]{CD1}. Hence we have
\begin{cor}
\label{semi}
Under the hypotheses of Corollary \ref{fin2}, the unitary flat bundle of the second Fujita decomposition is semi-ample.
\end{cor}

An example where this happens is the case where the fibers are hyperelliptic curves. First note that in this case $\mD=\mD^d$ and furthermore the strictness condition is trivially satisfied by Remark \ref{remstrict}. Hence we have

\begin{prop}
\label{hyper}
Let $X\to Y$ be a semistable fibration such that the general fiber $F$ is an hyperelliptic curve of genus $>2$.
Then $\mD$ is Massey trivial generated and  the unitary flat bundle $\sU$ of the Fujita decomposition is semiample. 
\end{prop}
\begin{proof}
Let $F$ be the general fiber and denote by $\xi\in H^1(T_F)$ the extension class attached to $F$.
Take $\eta_1,\eta_2\in \mD_y\subset H^0(F,\Omega^1_F)$ and construct their Massey product 
$s_1\wedge s_2\wedge\widehat{dy_k}|_{\frac{\partial}{\partial y}}$ which we denote by $m^k_\xi(\eta_1,\eta_2)$. This is obviously antisymmetric.

The hyperelliptic involution $\sigma$ gives a map $\sigma^*\colon H^0(\omega_F)\to H^0(\omega_F)$ which is the multiplication by $-1$. Hence we get
$$
\sigma^*m_{\xi}^k(\eta_1,\eta_2)=-m_{\xi}^k(\eta_1,\eta_2)
$$ while at the same time
$$
\sigma^*m_{\xi}^k(\eta_1,\eta_2)=m_{\xi}^k(-\eta_1,-\eta_2)
$$ This gives the desired vanishing of the Massey product hence $\mD$ is Massey trivial generated.
We can then apply Corollary \ref{semi} and conclude that the unitary flat bundle $\sU$ is semiample. 
\end{proof}

Note that this result can be generalized to every fibration with odd dimensional fibers such that the general fiber $F$ has an involution $\sigma$ such that $F/\sigma$ has $p_g = 0$ and such that $\mD$ is generated by anti-invariant 1-forms, cf. \cite[Proposition 5.15]{RZ4}.

\subsection{The monodromy of sub local systems of $\mD^d$}
We finish this section with an alternative tool for the study of the monodromy of $\mD^d$ and of its local subsystems.
The different approach comes from a version of the Castelnuovo-de Franchis theorem for $p$-forms, with $p>1$.
Let $X$ be a smooth variety, $w_1,\dots, w_l \in H^0 (X,\Omega^p_X)$, $l\geq p+1$, be linearly independent $p$-forms such that $w_i\wedge w_j=0$ (as an element of $\bigwedge^2\Omega^p_X$ and not of $\Omega_X^{2p}$) for any choice of $i,j=1,\dots, l$. These forms generate a subsheaf of $\Omega^p_X$ generically of rank $1$. Note that the quotients $w_i/w_j$ define a non-trivial global meromorphic function on $X$ for every $i\neq j$, $i,j=1,\dots, l$. By taking the differential $d (w_i/w_j)$ we then get global meromorphic $1$-forms on $X$. We assume that there exist $p$ of these meromorphic differential forms $d (w_i/w_j)$ that do not wedge to zero; if this is the case we call  the subset $\{w_1,\dots, w_l \}\subset H^0 (X,\Omega^p_X)$ $p$-strict. In new setting, this is of course the analogue of the strictness condition of Definition \ref{strict}. The following is Theorem 7.2 in \cite{RZ5}.
\begin{thm}
	\label{pcast}
	Let $X$ be an $n$-dimensional smooth variety and let $\{w_1,\dots, w_l \}\subset H^0 (X,\Omega^p_X)$ be a $p$-strict set. Then there exists a rational map $f\colon X\dashrightarrow Z$ over a $p$-dimensional smooth variety $Z$ of general type such that the $w_i$ are pullback of some holomorphic $p$-forms $\eta_i$ on $Z$, that is $w_i=f^*\eta_i$, where $i=1,\dots , l$.
\end{thm}
This result allows the study of local systems $\mW\leq\mD^d$ generated under monodromy by a vector space $W<\Gamma(A,\mD^d)$ where $W=\langle\eta_1,\dots,\eta_l\rangle$ and the sections $\eta_{i}$ are $d$-strict and admit liftings $s_i\in \Gamma(A,f_*\Omega_{X,d}^d)$ with $s_i\wedge s_j=0$ for every choice of $i,j$. Note that since the dimension of the fibers is $d$, the wedges $\eta_i\wedge\eta_j$ vanish by dimensional reasons, hence it makes sense to ask that the $\eta_{i}$ are $d$-strict. 
\begin{thm}
	Let $W=\langle\eta_1,\dots,\eta_l\rangle$ and assume that the sections $\eta_{i}$ are $d$-strict and admit liftings $s_i\in \Gamma(A,f_*\Omega_{X,d}^d)$ with $s_i\wedge s_j=0$ for every choice of $i,j$. Then $W$ generates a local system $\mW$ with finite monodromy group $G_\mW$.
\end{thm}
\begin{proof}
	The proof is similar to Theorem \ref{monfin}. 
	
	Since by hypothesis the $\eta_i$ are $d$-strict, it is easy to see that the $s_i$ are also $d$-strict. Furthermore, similarly to Therorem \ref{castmassey}, also in the non-compact case we can apply Theorem \ref{pcast} because the forms $s_i$ are closed $d$-forms.
	Hence after taking the appropriate covering we obtain meromorphic maps that, restricted to the general fiber, define the following set
	$$
	\hat{\sK}=\{k_g\colon F_0 \dashrightarrow Z\mid g\in G_\mW\}.
	$$
	By strictness the maps $k_g$ are dominant and, as in the previous cases, the monodromy group acts faithfully on this set; Lemma \ref{formula1} and Lemma \ref{faithful}  can be easily adapted also in the case of $d$-forms. 
	
	The set $\hat{\sK}$ is finite thanks to \cite{kob} and we conclude.
\end{proof}
\section{Massey products of higher degree}
\label{masseybig}
So far, when working with Massey products, we have always considered a wedge of $d+1$ differential forms, where $d=n-m$ is the relative dimension of $f\colon X\to Y$, see Definition \ref{mtrivial} and following. The goal of this section is to analyze what happens when we consider the analogous construction with a wedge product of a number of sections greater than $d+1$. We point out that in the usual context of \cite{RZ1}, \cite{PT}, \cite{RZ4}, that is when the base $Y$ is a curve, we have that $d+1=n=\dim X$ hence this analysis is only meaningful when $Y$ is not a curve. 

Consider an integer $k$ with $d+1\leq k\leq n$.
Given $k$ linearly independent sections $\eta_1,\ldots,\eta_{k}\in \Gamma(A,K_{\partial})$, we choose appropriate liftings in $\Gamma(A,f_*\Omega_{X}^1)$ by Lemma \ref{split}. We take the wedge of these forms on $X$ via
\begin{equation}
\bigwedge^{k} f_*\Omega^1_X\to f_*\Omega^{k}_X
\end{equation}
 to obtain an element of $f_*\Omega^{k}_X$.

%\begin{equation}
%\label{agg}
%\bigwedge^{d+1} K_{\partial}\to f_*\Omega^{d+1}_X
%\end{equation}

%From the sections $\eta_1,\ldots,\eta_{d+1}$ we also consider the wedge products $\eta_1\wedge\dots\wedge\widehat{\eta_i}\wedge \dots\wedge\eta_{d+1}$ for $i=1,\ldots,{d+1}$ and we work similarly taking liftings in  
%\begin{equation}
%%\label{wedge}
%\bigwedge^{d} f_*\Omega^1_X\to \bigwedge^{d} K_{\partial}
%\end{equation} and then their images in $f_*\Omega^{d}_{X}$.
%
% splitting 
%\begin{equation}
%\label{split3}
%\bigwedge^{k} K_{\partial}\to \bigwedge^{k} f_*\Omega^1_X.
%\end{equation}
%which is the analogue of (\ref{split2}).
%Once again this map is not unique but this choice is not important in the following.
%Taking the wedge of forms on $X$ we obtain the composition
%\begin{equation}
%\label{agg2}
%\bigwedge^{k} K_{\partial}\to \bigwedge^{k} f_*\Omega^1_X\to f_*\Omega^{k}_X
%\end{equation} which is the analogue of (\ref{agg}).
\begin{defn}
\label{kmas}
In complete analogy with Definition \ref{mtrivial}, the Massey product of $\eta_1,\ldots,\eta_{k}$, $d+1\leq k\leq n$, is the section of $f_*\Omega^{k}_X$ computed via this construction. We call it $k$-Massey product if we want to highlight the number of starting sections.
\end{defn}
On the other hand, it is not convenient to define Massey triviality following Definition \ref{mastriv}, because in this case it turns out to be quite convoluted. From Remark \ref{eccoperche} we have the following easier, but equivalent, definition
\begin{defn}
	\label{kmt}
	We say that $\eta_1,\ldots,\eta_{k}$ are $k$-Massey trivial if there exist liftings ${s}_1,\dots,{s}_{k}$ of the $\eta_i$ such that ${s}_1\wedge\dots\wedge{s}_{k}=0$. We say that a subspace $W<\Gamma(A,\mD)$ is $k$-Massey trivial if any $k$-uple of linearly independent sections in $W$ is  $k$-Massey trivial.
\end{defn}

In the following we make the usual assumption of strictness, that is no wedge of less than $k$ of these liftings is zero.

% in particular the $\omega_I$ are not zero when restricted to the general fiber, and furthermore that all the possible $l$-Massey products are also different from zero for $l<k$. This is a natural hypothesis because otherwise, if for example one of the $m^{l}_{J}$ vanish for $l<k$, then we have a subset of the $\eta_i$ which is $l$-Massey trivial and studying $k$-Massey triviality is superfluous.

%Comparing this definition with Equation (\ref{aggz1}) the similarity is clear and hence it is not difficult to see that results as Proposition \ref{aggzero1} holds using the same reasoning. Therefore we only give the statement without proof.
%\begin{prop}
%\label{aggzero2}
%If $\eta_1,\ldots,\eta_{k}$ are $k$-Massey trivial we can choose liftings $\widetilde{s}_1,\dots,\widetilde{s}_{k}$ of the $\eta_i$'s such that $\widetilde{s}_1\wedge\dots\wedge\widetilde{s}_{k}=0$.
%\end{prop}

We are mostly concerned with Massey products built from sections $\eta_i\in \Gamma(A,\mD)$.
From the definition it is clear that k-Massey triviality is strictly related to the Castelnuovo-de Franchis hypothesis. As a first example, using the same notation of Theorem \ref{castmassey} for the base change associated to a subgroup $K$ of $H_W$ we have the following result.

\begin{thm}
\label{castmassey2}
Let $X\to Y$ be a semistable fibration with $\dim X=n$, $\dim Y=m$. Let $A\subset Y$ be an open subset and $W\subset \Gamma(A,\mD)$ a $k$-Massey trivial strict subspace.
Then $X_K$ has a morphism $h_K\colon X_K\to Z$ over a normal $(k-1)$-dimensional variety of general type $Z$ such that $W\subset h_K^*(H^0(Z,\Omega^1_Z))$. Furthermore if $W$ is a maximal $k$-Massey trivial subspace we have the equality $W= h_K^*(H^0(Z,\Omega^1_Z))$.
\end{thm}
\begin{proof}
As we have seen before, the key point is to use the Castelnuovo-de Franchis Theorem \ref{cas2}.
Up to base change $Y_{K}\to Y$, the sections $\eta_i$, basis of $W$, produce global sections. We can choose liftings for the $\eta_i$ that wedge to zero. 
%Note that this follows immediately from Definition \ref{kmt} if $\dim W=k$, otherwise if $\dim W>k$, we need an induction argument similar to the one in Proposition \ref{wedge0}. 
Then we conclude with the Castelnuovo-de Franchis Theorem \ref{cas2}.

Once again note that even in the case $X_{K}$ not compact, the proof works because the sections $\eta_i$ are in $\mD$ and liftings of sections of $\mD$ are closed by definition.
\end{proof}

This result tells us that, by the same construction of Section \ref{sez5}, we obtain maps $X_\mW\to Z$ and an action of the monodromy group $G_\mW$ on the sets 
$$
\mathcal{H}=\{h_g\colon X_\mW\to Z\mid g\in G_\mW\}
$$ and
$$
\mathcal{K}=\{k_g\colon F_0\to Z\mid g\in G_\mW\}.
$$
The key difference is now that the maps $k_g$ are not surjective for dimensional reasons, hence $\sK$ is not necessarily finite. It follows that, while $G_\mW$ still injects in $\text{Aut}(\sK)$, we cannot immediately conclude that it is finite.

Hence we study the action of $G_\mW$ on $\sH$ instead and we have the following result
\begin{thm}
	\label{chiusozar}
Let $f \colon X \to Y$ be a semistable fibration and let $\mW\leq\mD$ a local system generated by a maximal strict $k$-Massey trivial subspace. Assume also that $Y_\mW\subset \overline{Y_\mW}$ is a Zarisky open subset of a compact variety.
	Then the associated monodromy group $G_\mW$ is finite.
\end{thm}
\begin{proof}
	As anticipated, we consider the action of $G_\mW$ on $\sH$.
We note that, under our hypothesis, $\sH$ is finite by \cite{kob}, see also \cite[Section 2.2]{ZL}, hence we only need to show that the action is faithful. 

We proceed as in Section \ref{sez5}; the same computations of Lemma \ref{formula1} and \ref{faithful} in fact tells us that if $g\in G_\mW$, $g\neq e$ there exists an element $g'$ such that $g\cdot h_{g'}\neq h_{g'}$.
\end{proof}

\subsection{A new insight on the Adjoint Theorem}
As a final step to show the close relation between the theory of Massey products and the Castelnuovo-de Franchis theorem we give a new insight on the classical Adjoint theorem of \cite{PZ}, \cite{RZ1} as follows.

Take a general point $y\in Y$, we  restrict Sequence (\ref{diffrel}) to the fiber $F$ over $y$ and obtain, after identification of $T_{Y,y}^\vee$ with $\mC^m$, the exact sequence
\begin{equation}
\label{seqm}
0\to \sO_F^m\to \Omega^1_{X|F}\to \Omega^1_F\to 0.
\end{equation}  This sequence is extensively studied in \cite{RZ1} with $m=1$.
In that paper we consider $\eta_1,\dots,\eta_{n}$ global $1$-forms on $F$ that can be lifted to $X$ and we call $D$ the divisor on $F$ which is the fixed part of the linear system of the sections $\eta_1\wedge\dots\wedge\widehat{\eta_i}\wedge\dots\wedge\eta_n$. Note that $\eta_1\wedge\dots\wedge\widehat{\eta_i}\wedge\dots\wedge\eta_n$ is just the restriction on $F$ of the $\omega_i$ in Definition \ref{omegai}. One of the main results of the paper is the following \cite[Theorem A]{RZ1}.
\begin{thm}
\label{teoremaA}
Let $\xi\in 
{\rm{Ext}}^1(\Omega^1_F,\mathcal{O}_F)$ be the extension class of 
the sequence
\begin{equation}
	\label{solitaff}
0\to \sO_F\to \Omega^1_{X|F}\to \Omega^1_F\to 0.
\end{equation}
Let $\omega$ be a Massey product  
associated to the sections $\eta_i$. If $\omega$ is Massey trivial then 
$\xi\in\Ker({\rm{Ext}}^1(\Omega_F^1,\mathcal{O}_F)\to 
{\rm{Ext}}^1(\Omega_F^1(-D),\mathcal{O}_F))$.
\end{thm}
This result has been used mainly to prove infinitesimal Torelli type theorems.

In our new setting we use the same definition of $D$, that is the fixed part of the sections $\eta_{j_1}\wedge\dots\wedge\widehat{\eta_{j_i}}\wedge\dots\wedge\eta_{j_d}$ with $\{j_1,\dots, j_d\}\subset\{1,\dots,k\}$, and we prove the following result.
\begin{thm}
\label{aggiuntanuovo}
Consider Sequence (\ref{seqm}) and $\eta_1,\dots,\eta_k$ global sections of $H^0(\Omega_F^1)$ which are in $\mD$. If the $\eta_i$ are $k$-Massey trivial, then there exists a torsion free sheaf $\sF$  of rank $k-1\choose d$ contained in $\Omega^d_{X|F}$ whose image $\sL$ in $\omega_{F}$ satisfies $\sL^{\vee\vee} =\omega_{F}(-D')$, with $D'\leq D$ an effective divisor.
If we take $k=d+1$, $\sF$ is isomorphic onto the image, hence $\sF^{\vee\vee} \cong\omega_{F}(-D')$.
\end{thm}
\begin{proof}
Recall that $y$ is the base point of $Y$ corresponding to $F$. Take a small polydisk $\Delta$ around $y$ and take liftings of the $\eta_i$ in $\Gamma(\Delta, f_*\Omega^1_{X,d})$. As always these liftings are closed and by Massey triviality we can assume that their wedge is zero. By Castelnuovo-de Franchis we find a morphism $h\colon f^{-1}(\Delta)\to Z$ which if restricted to the fiber $F$ gives a morphism $F\to Z$.

We can then define $\sF=(h^*\Omega^d_Z)_{|F}$. Since $(h^*\Omega^d_Z)_{|F}$ is a subsheaf of $\Omega^d_{X|F}$, $\sF$ fits into the following diagram
\begin{equation}
\xymatrix{
\Omega^d_{X|F}\ar[r]&\omega_F\\
\sF\ar@{^{(}->}[u]\ar[ur]&
}
\end{equation}
Since the $\eta_i$ come from $Z$, the sections  $\eta_{j_1}\wedge\dots\wedge\widehat{\eta_{j_i}}\wedge\dots\wedge\eta_{j_d}$ are in the image of $\sF$ and  they generate $\omega_{F}$ outside of $D$ and some loci of codimension $\geq2$. Note that in principle $\sF$ may contain also other sections, hence the image of $\sF$ actually generates a sheaf $\sL=\omega_{F}(-D')\otimes \sI$, with $D'\leq D$.
Taking the double dual gives the thesis.

The statement for $k=d+1$ follows from noticing that in this case $Z$ has dimension $d$ and $\sF$ is the pullback of its canonical sheaf. Hence $\sF$ and its image have the same rank and the kernel would be a torsion subsheaf, which is not possible.
\end{proof}

Now if consider the wedge product of Sequence (\ref{seqm}) we obtain an exact sequence of the form
\begin{equation}
	\label{seqmwe}
	0\to \sK\to \Omega^d_{X|F}\to \omega_F\to 0,
\end{equation} hence we have defined a map
\begin{equation}
	{\rm{Ext}}^1(\Omega^1_F,\sO^m_F)\to {\rm{Ext}}^1(\omega_F,\sK).
\end{equation} 
From Theorem \ref{aggiuntanuovo} we easily recover a slightly more general version of Theorem \ref{teoremaA}.

\begin{cor}
\label{cork}
With the hypothesis of Theorem \ref{aggiuntanuovo} and $k=d+1$ we have that $\xi\in {\rm{Ext}}^1(\Omega^1_F,\sO^m_F)$ corresponding to Sequence (\ref{seqm}) is in the kernel of 
\begin{equation}
{\rm{Ext}}^1(\Omega^1_F,\sO^m_F)\to {\rm{Ext}}^1(\omega_F(-D'),\sK).
\end{equation} If furthermore $m=1$, $\xi$ is in the kernel of 
\begin{equation}
	{\rm{Ext}}^1(\Omega^1_F,\sO_F)\to {\rm{Ext}}^1(\Omega^1_F(-D'),\sO_F).
\end{equation} 
\end{cor}
\begin{proof}
By the previous Theorem we know that $\sF^{\vee\vee}$ is isomorphic to $\omega_F(-D')$. This gives a splitting of the sequence 
\begin{equation}
0\to \sK\to \sE\to \omega_F(-D')\to 0
\end{equation} which corresponds to the image of $\xi$ via the composition 
\begin{equation}
\text{Ext}^1(\Omega^1_F,\sO^m_F)\to \text{Ext}^1(\omega_F,\sK)\to \text{Ext}^1(\omega_F(-D'),\sK).
\end{equation} 
For the second assertion, note that in the case $m=1$ we have $\sK=\Omega^{d-1}_F$ and $\text{Ext}^1(\omega_F(-D'),\Omega^{d-1}_F)\cong \text{Ext}^1(\Omega^1_F(-D'),\sO_F)$.
\end{proof}
\begin{rmk}
Note that since $D'\leq D$, from the inclusion 
$$
\Omega^1_F(-D)\subset \Omega^1_F(-D')\subset \Omega^1_F
$$ we have a map of extension groups 
\begin{equation}
\text{Ext}^1(\Omega^1_F,\sO_F)\to \text{Ext}^1(\Omega^1_F(-D'),\sO_F)\to \text{Ext}^1(\Omega^1_F(-D),\sO_F).
\end{equation} Hence if the image of $\xi$ is is zero in $\text{Ext}^1(\Omega^1_F(-D'),\sO_F)$ by this Corollary, it is also zero in $\text{Ext}^1(\Omega^1_F(-D),\sO_F)$; this is why this result is a slightly more general version of the Adjoint theorem \ref{teoremaA}. 

The main point of these results is that they highlight the fact that the splitting obtained in Theorem \ref{teoremaA} is due to the presence of the variety $Z$ given by the Castelnuovo-de Franchis theorem.
\end{rmk}

%%%%%%%%%%%%%%%%%%%%%%%%%%%%%%%%%%%%%%%%%%%%%%%%%%%%%%%
%%%%%%%%%%%%%%%%%%%%%%%%%%%%%%%%%%%%%%%%%%%%%%%%%%%%%%%
%%%%%%%%%%%%%%%%%%%%%%%%%%%%%%%%%%%%%%%%%%%%%%%%%%%%%%%

\end{document}